\documentclass{amsart}
\usepackage{amsmath,amsthm,amsfonts,amssymb}
\usepackage{mathrsfs}
\usepackage{stmaryrd}
\usepackage{mathtools}
\usepackage{subfiles}
\usepackage[sans]{dsfont}
\usepackage{enumerate}
\usepackage{fullpage}
\usepackage{tikz}
\usepackage{float}
\usepackage{cancel}
\usepackage{xcolor}
\usepackage{setspace}
\usepackage{calligra}
\usepackage{MnSymbol}
\RequirePackage{doi}
\usepackage{hyperref}
\usepackage[all]{xy}
\usepackage[capitalize]{cleveref}

\usetikzlibrary{shapes.geometric}
\usetikzlibrary{calc}

\newtheorem*{clm*}{Claim}

\newtheorem{thm}{Theorem}[section]
\newtheorem{lem}[thm]{Lemma}
\newtheorem{cor}[thm]{Corollary}
\newtheorem{prop}[thm]{Proposition}

\theoremstyle{definition}
\newtheorem{defn}[thm]{Definition}

\newtheorem{axiom}[thm]{Axiom}

\theoremstyle{remark}
\newtheorem{rmk}[thm]{Remark}

\newtheorem{ntn}[thm]{Notation}

\newtheorem{cons}[thm]{Construction}

\newtheorem{ques}[thm]{Question}

\numberwithin{equation}{section}

\DeclareMathOperator{\aim}{Im}

\DeclareMathOperator{\ahom}{Hom}
\DeclareMathOperator{\aker}{Ker}
\DeclareMathOperator{\ann}{Ann}
\DeclareMathOperator{\ass}{Ass}

\DeclareMathOperator{\coker}{Coker}

\DeclareMathOperator{\ext}{Ext}
\DeclareMathOperator{\fr}{Frac}

\DeclareMathOperator{\SHom}{\mathscr{H}\text{\kern -3pt {\calligra\large om}}\,}
\DeclareMathOperator{\SKer}{\mathscr{K}\text{\kern -3pt {\calligra\large er}}\,}
\DeclareMathOperator{\SDer}{\mathscr{D}\text{\kern -2pt {\calligra\large er}}\,}
\DeclareMathOperator{\SEnd}{\mathscr{E}\text{\kern -3pt {\calligra\large nd}}\,}
\DeclareMathOperator{\SSym}{\mathscr{S}\text{\kern -3pt {\calligra\large ym}}\,}

\newcommand{\mbb}[1]{\mathbb{#1}}

\newcommand{\QQ}{\mbb{Q}}
\newcommand{\NN}{\mbb{N}}

\newcommand{\mf}[1]{\mathfrak{#1}}
\newcommand{\mm}{\mf{m}}

\newcommand{\mrm}[1]{\mathrm{#1}}

\newcommand{\mcal}[1]{\mathcal{#1}}
\newcommand{\cA}{\mcal{A}}
\newcommand{\cB}{\mcal{B}}
\newcommand{\cC}{\mcal{C}}

\newcommand{\OO}{\mcal{O}}
\newcommand{\cR}{\mcal{R}}

\newcommand{\seq}[2][x]{#1_1,...,#1_{#2}}

\newcommand{\rdf}{\mathds{R}}

\DeclarePairedDelimiter\ps{\llbracket}{\rrbracket}
\DeclarePairedDelimiter\set{\{}{\}}
\DeclarePairedDelimiter\abs{\lvert}{\rvert}

\makeatletter
\renewcommand*{\eqref}[1]{%
  \hyperref[{#1}]{\textup{\tagform@{\ref*{#1}}}}%
}
\makeatother

\definecolor{zhan}{rgb}{0.0, 0.44, 1.0}

\hypersetup{
  colorlinks = true,
  linkcolor = red,
  anchorcolor = zhan,
  citecolor = blue,
  filecolor = cyan,
  menucolor = red,
  runcolor = cyan,
  urlcolor = magenta,
}

\newcommand{\mbc}{\cB^{\circ}}
\newcommand{\vr}[1]{\mcal{#1}^{\circ}}
\newcommand{\kc}{K^{\circ}}
\newcommand{\nseq}[3][x]{#1_{#2},\ldots,#1_{#3}}
\newcommand{\epg}{\left( pg\right)^{\varepsilon}}
\newcommand{\pir}[1]{#1^{1/p^{\infty}}}
\newcommand{\cp}[1]{\widehat{#1^+}}
\newcommand{\dsm}[2][R]{#1^{\oplus #2}}
\newcommand{\alm}[1]{\overset{a}{#1}}
\newcommand{\rpptl}{R^+(t_{\lambda}:\lambda\in \Lambda)}
\newcommand{\rpbtl}{R^+[t_{\lambda}:\lambda\in \Lambda]}
\newcommand{\rptl}{R(t_{\lambda}:\lambda\in \Lambda)}

\newcommand{\seqp}[3][N]{#2_1^{#1},\ldots,#2_{#3}^{#1}}

\DeclareMathOperator{\spa}{Spa}
\DeclareMathOperator{\sym}{sym}
\DeclareMathOperator{\pc}{pc}
\newcommand{\epf}{\mathsf{epf}}
\newcommand{\wepf}{\mathsf{wepf}}
\newcommand{\rof}{\mathsf{r1f}}
\newcommand{\cl}{\mathsf{cl}}

\begin{document}

\title{Closure Operations in Complete Local Rings of Mixed Characteristic}
\author{Zhan Jiang}
\date{\today}
\address{
Zhan Jiang \\
  Department of Mathematics \\
  University of Michigan \\
  Ann Arbor, MI 48109--1043 \\
  USA
}
\email{zoeng@umich.edu}
\urladdr{\url{www-personal.umich.edu/~zoeng/}}
\thanks{The author was partly supported by NSF grant DMS-1902116}

\begin{abstract}
  The extended plus ($\epf$) closure and the rank 1 ($\rof$) closure are two closure operations introduced by Raymond C. Heitmann for rings of mixed characteristic. Recently, he and Linquan Ma proved that the $\epf$ closure satisfies the usual colon-capturing property under mild conditions. In this paper, we extend their result and prove that the $\epf$ closure satisfies what we call the $p$-colon-capturing property. Based on that, we define a new closure notion called ``weak $\epf$ closure,'' and prove that it satisfies the generalized colon-capturing property and some other colon-capturing properties. This gives a new proof of the existence of big Cohen-Macaulay algebras in the mixed characteristic case. We also show that any module-finite extension of a complete local domain is $\epf$-phantom, which generalizes a result of Mel Hochster and Craig Huneke about ``phantom extensions.'' Finally, we prove some related results in characteristic $p$.
\end{abstract}

\maketitle
% \tableofcontents

\section{Introduction}\label{sn-introduction}
In equal characteristic $p>0$, tight closure was introduced by Mel Hochster and Craig Huneke in 1986, and turned out to be a very powerful tool. For example, it can be used to prove the existence of balanced big Cohen-Macaulay algebras for rings containing a field. In order to have a similar tool in mixed characteristic, Raymond Heitmann introduced four closure operations for mixed characteristic rings \cite{Heitmann2001}. He also proved one of them, the $\epf$ closure, satisfies the (usual) colon-capturing \cite[Theorem 3.7]{Heitmann2002} for rings of mixed characteristic of dimension at most $3$. Based on this result, he was able to prove the direct summand conjecture in that case.

Geoffrey Dietz and Rebecca R.G. studied the relation between the existence of balanced big Cohen-Macaulay
algebras (modules) and closure operations. Dietz introduced seven axioms (see \cite[Axiom
1.1]{Dietz2010} and \cref{app-axm-DietzAxioms} in \cref{sn-Preliminaries}). We call a closure operation a \emph{Dietz closure} if it satisfies all of Dietz's axioms. Dietz proved that a local domain $R$ has a Dietz closure if and only if it has a balanced big Cohen-Macaulay module. In \cite{R.G.2018}, R.G. introduced a new axiom called the \emph{algebra axiom}, and proved that the existence of a Dietz closure satisfying the algebra axiom is equivalent to the existence of a balanced big Cohen-Macaulay algebra. Recently, the existence of balanced big Cohen-Macaulay algebras in mixed characteristic was completely solved by Yves Andr\'e using perfectoid techniques \cite{Andre2018d}. Inspired by this result, Raymond Heitmann and Linquan Ma were able to prove that the $\epf$ closure satisfies the (usual) colon-capturing condition \cite[Corollary 3.11]{Heitmann2018}.

The purpose of our paper is to develop a new closure operation called $\wepf$ (\cref{m-defn-WEPF}), and prove that it is a Dietz closure satisfying the algebra axiom (\cref{app-thm-WEPFDietz}). This gives a new proof of the existence of big Cohen-Macaulay algebras (modules). We achieve this by proving a strong property about $\epf$ closure of ideals generated by part of system of parameters (\cref{m-thm-pColonCap}), which we call \emph{$p$-colon-capturing} (\cref{m-defn-pcc}). This property generalizes some results in \cite{Heitmann2018}. We point out that our $p$-colon-capturing property can also be used to prove that $\rof$ is a Dietz closure satisfying the algebra axiom. So far as we know, the problem whether $\epf$ is a Dietz closure remains open.

We also prove that every module-finite extension of a complete
local domain of mixed characteristic with an $F$-finite residue field is $\epf$-phantom (\cref{m-thm-EPFPhantom}). This result,
together with Heitmann and Ma's result \cite[Theorem
3.19]{Heitmann2018}, implies the direct summand conjecture. We make great use of techniques from  Yves Andr\'e's \cite{Andre2018l} and
Bhargav Bhatt's results \cite{Bhatt2017}. We also prove a property (\cref{p-thm-ColonCap}) similar to $p$-colon-capturing in the positive characteristic case. This is a completely new phenomenon about tight closure.

This paper is organized as follows: \cref{sn-Preliminaries} collects some
preliminaries on basic notations and techniques. In \cref{sn-p-ColonCapturing}, we prove the $p$-colon-capturing property using the perfectoid Abhyankar lemma (\cref{m-thm-PerfectoidAbhyankar}). This is one of our main results (\cref{m-thm-pColonCap}). In \cref{sn-WeakEPFClosure} we introduce our
new closure operation $\wepf$, and prove that it is a Dietz closure satisfying the algebra axiom (\cref{m-thm-WepfGCC}, \cref{app-thm-WEPFDietz}). In \cref{sn-PhantomExtensions} we show that module-finite extensions are $\epf$-phantom (\cref{m-thm-EPFPhantom}) using techniques different from those in \cref{sn-p-ColonCapturing}. Finally, in \cref{sn-p-PositiveCharCase}, we study the behaviour of regular sequences on some non-noetherian rings (\cref{p-thm-RSonT}) and prove results similar to $p$-colon-capturing (\cref{p-thm-ColonCap}) in
characteristic $p$.

\subsection*{Acknowledgments}
I would like to heartily thank Mel Hochster for directing me to this question and offering helpful discussions and advice.
I am deeply indebted to the work of Raymond Heitmann and Linquan Ma on the $\epf$ closure \cite{Heitmann2018}, and thank them both for fruitful conversations on the first few drafts of this paper. Raymond Heitmann also pointed out the unnecessary assumption of uncountably infinite residue field in \cref{p-thm-ColonCap}.
I want to thank Devlin Mallory, Takumi Murayama, Rebecca R.G. and Kazuma Shimomoto for useful conversations and comments on this paper. I also want to thank the anonymous referee for the careful reading of the paper and many insightful comments and suggestions.

\section{Preliminaries} \label{sn-Preliminaries}
Throughout this paper, $p$ will always be a positive prime integer. A
ring of mixed characteristic $p$ is a ring $R$ of characteristic $0$
with $p$ in every maximal ideal of $R$. We will work with a complete
local ring of mixed characteristic $p$ in all sections except \cref{sn-p-PositiveCharCase}. We will also use the following notation.
\begin{ntn}\label{i-ntn-plus}
Let $R$ be a domain. By $R^+$ we mean an absolute integral closure of a domain $R$, i.e., the integral closure of $R$ in an algebraic closure of its fraction field $\fr(R)$. For any $R$-module $M$, we write $M^+:=R^+\otimes_R M$. For any submodule $W\subseteq M$, we write $W^+$ for the tensor product $R^+\otimes_R W$, and $\aim(W^+\to M^+)$ for the image of the map $R^+\otimes_R W\to R^+\otimes_R M$ in $M^+$.
\end{ntn}

Note that in the literature, the notion $I^+$ means plus closure of $I$, i.e., $IR^+\cap R$. Since we are using neither the plus closure nor the notation $I^+$, there should be no confusion.

\subsection{Closure operations in mixed characteristic}
 Let us recall the definition of Heitmann's two closure operations, $\epf$ and $\rof$, below.

\begin{defn}\label{i-def-epf}
  Let $R$ be an integral domain of mixed characteristic $p$ and let $I$ be an ideal of $R$. Then an element $x\in R$ is in the (full) extended plus closure of $I$, i.e., $x\in I^{\epf}$, provided there exists $c\in R-\set{0}$ such that for every positive rational number $\varepsilon$ and every positive integer $N$, $c^{\varepsilon}x\in (I,p^N)R^+$. The element $x$ is in the (full) rank $1$ closure of $I$, i.e., $x\in I^{\rof}$, if for every rank one valuation $\nu$ on $R^+$, every positive integer $N$, and every positive rational number $\varepsilon$, there exists $d\in R^+-\set{0}$ with $\nu(d)< \varepsilon$ such that $d x\in (I,p^N)R^+$.
\end{defn}
From the definition above, we immediately see that the $\epf$ closure is always contained in the $\rof$ closure, i.e., $I^{\epf}\subseteq I^{\rof}$ for any ideal $I\subseteq R$. We also note that there is a natural generalization of these definitions to (finitely generated) modules. See also \cite[Definition 7.1]{R.G.2016}. We include the definition below.

\begin{defn}\label{i-def-EpfModule}
  Let $R$ be an integral domain of mixed characteristic $p$ and let $W\subseteq M$ be finitely generated $R$-modules. Let $M^+,W^+$ be the notation in \cref{i-ntn-plus}, and let $u$ be an element of $M$. Then $u\in M$ is in the $\epf$ closure of $W$ if there is some $c\in R-\set{0}$ such that for any $\varepsilon\in\QQ^+,N\in\NN$ we have
  \[
    c^{\varepsilon}\otimes u \in \aim \left( W^+\rightarrow M^+ \right)+p^N M^+.
  \]
  Moreover, $u$ is in the $\rof$ closure of $W$ if for every rank one valuation $\nu$ on $R^+$, every positive integer $N$, and every positive rational number $\varepsilon$, there exists $d\in R^+-\set{0}$ with $\nu(d)< \varepsilon$ such that
  \[
    d\otimes u \in \aim \left( W^+\rightarrow M^+ \right)+p^N M^+.
  \]
\end{defn}

\subsection{Closure axioms}
Here we present the seven axioms defined by Dietz in \cite{Dietz2010}, together with the algebra axiom defined by R.G. in \cite[Axiom 3.1]{R.G.2018}.
\begin{axiom}\label{app-axm-DietzAxioms}
  Let $(R,\mm)$ be a complete local domain possessing a closure operation $\cl$. Let $Q,M$ and $W$ be arbitrary finitely generated $R$-modules with $Q\subseteq M$.
  \begin{enumerate}
  \item (Extension) $Q_M^{\cl}$ is a submodule of $M$ containing $Q$.
  \item (Idempotence) $(Q_M^{\cl})_M^{\cl}=Q_M^{\cl}$.
  \item (Order-preservation) If $Q\subseteq M\subseteq W$, then $Q_W^{\cl}\subseteq M_W^{\cl}$.
  \item (Functoriality) Let $f:M\rightarrow W$ be a homomorphism. Then $f(Q_M^{\cl})\subseteq (f(Q))_W^{\cl}$.
  \item (Semi-residuality) If $Q_M^{\cl}=Q$, then $0^{\cl}_{M/Q}=0$.
  \item The maximal ideal $\mm$ and the zero ideal $0$ are closed.
  \item (Generalized Colon-capturing) Let $\seq{k+1}$ be a partial system of parameters for $R$ and
    let $J=(\seq{k})$. Suppose that there exists a surjective homomorphism
    $f:M\rightarrow R/J$ such that $f(v) = x_{k+1}+J$. Then
    \[
      (Rv)^{\cl}_M \bigcap \aker f \subseteq (Jv)_M^{\cl}.
    \]
  \end{enumerate}
  Next is R.G.'s algebra axiom.
  \begin{enumerate}
    \setcounter{enumi}{7}
  \item If $R\overset{\alpha}{\to} M,1\mapsto e_{1}$ is $\cl$-phantom, then
    the map $\alpha':R\to \sym^{2}(M),1\mapsto e_{1}\otimes e_{1}$ is
    $\cl$-phantom.
  \end{enumerate}
\end{axiom}
The seventh axiom (the generalized colon-capturing axiom) is also equivalent to the following axiom if the closure operation satisfies the other six Dietz axioms. See \cite[Lemma 1.3]{Dietz2010}.
\begin{axiom}\label{axiom-alternate}
  Let $R$ be a complete local domain possessing a closure operation
  $\cl$. Assume that $\dim R=d$. Let $\seq{k+1}$ be a partial system of parameters for $R$ where $0\leqslant k<d$ and
  let $J=(\seq{k})$ ($J=0$ if $k=0$). Suppose that there exists a homomorphism
  $f:M\rightarrow R/J$ such that $f(v) = x_{k+1}+J$. Then
  \[
    (Rv)^{\cl}_M \bigcap \aker f \subseteq (Jv)_M^{\cl}.
  \]
\end{axiom}
Since both the $\epf$ and $\rof$ closures satisfy the first 6 axioms \cite[Section 7]{R.G.2016}, we have no trouble using this equivalent form. 

This axiom (7) in \cref{app-axm-DietzAxioms} and its alternate form \cref{axiom-alternate} are rather subtle in comparison with the other axioms. It is not even obvious that tight closure satisfies this condition. Axiom (7) implies that the closure operation is ``big enough'' without being ``too big''. Note that integral closure for ideals can be extended to modules such that it satisfies axioms (1) - (6) and ordinary colon capturing. % Integral closure: it's true if base change to any valuation domain. (DVR if noetherian).

 However, generalized colon-capturing is the most critical in the proof by Dietz that the existence of a closure operation satisfying axioms (1) - (7) in \cref{app-axm-DietzAxioms} is equivalent to the existence of a balanced big Cohen-Macaulay module. Dietz also proved that the usual notion of colon-capturing follows from it \cite[Proposition 1.4]{Dietz2010}.

\subsection{Phantom extensions}
The notion of phantom extensions was first introduced by Hochster and Huneke in \cite[Section 5]{Hochster1994a} in order to produce a new proof for the existence of big Cohen-Macaulay modules. In the same paper, they also proved that every module-finite extension of a reduced ring of positive characteristic is a phantom extension \cite[Theorem 5.17]{Hochster1994a}. The generalized notion related to a closure operation was introduced by Dietz \cite[Definition 2.2]{Dietz2010}, which we record below.

An exact sequence $0\rightarrow R\overset{\alpha}{\rightarrow} M\rightarrow Q\rightarrow 0$ determines an element $\epsilon$ in $\ext_R^1(Q,R)$ via the Yoneda correspondence. If $P_{\bullet}$ is a projective resolution of $Q$ consisting of finitely generated projective modules $P_i$, then $\epsilon$ is a cocycle element in $\ahom_R(P_1,R)$. We call $\epsilon$ $\cl$-phantom if $\epsilon$ is in $\aim\left(\ahom_R(P_0,R)\rightarrow\ahom_R(P_1,R)\right)^{\cl}_{\ahom_R(P_1,R)}$.

\begin{rmk}
  This is different from requiring that $\epsilon$ as an element of $\ext_R^1(Q,R)$ is in the $\cl$ closure of $0$. Because $\ext_R^1(Q,R)$ is a submodule of $\ahom_R(P_1,R)/\aim\left(\ahom_R(P_0,R)\rightarrow\ahom_R(P_1,R)\right)$, the $\cl$-closure of $0$ in the latter one could be potentially larger than the closure in the former one.
\end{rmk}

We note that this definition is independent of the choice of the resolution of $Q$. For proofs, see \cite[Discussion 2.3]{Dietz2010}.

\subsection{Almost mathematics}
The language of almost mathematics is carefully studied in \cite{Gabber2003}. We will not use the full strength of that. The setup of almost mathematics is given by a ring $A$ together with an $A$-flat ideal $I$ such that $I^2=I$. The situation where almost mathematics is involved in this paper usually occurs over an algebra $A$ with an $A$-flat ideal $I=(\pir{c})A$, where $(\pir{c})A$ means the ideal generated by a compatible system of $p$-power roots of $c$, i.e., $(c^{1/p},c^{1/p^2},\ldots)A$.
This situation can be explained explicitly: let $M$ be an $A$-module. An element $u\in M$ is $I$-almost zero, i.e., $u\alm{=}0$, if and only if $c^{1/p^k}u=0$ for any $k\in \NN$, or, equivalently, $Iu=0$. An element $u$ is $I$-almost in a submodule $N$ of $M$, i.e., $u\alm{\in} N$ if its image in $M/N$ is almost zero. A submodule $N_1$ is $I$-almost in $N_2$, i.e., $N_1\alm{\subseteq} N_2$, if every element in $N_1$ is $I$-almost in $N_2$. Two submodules $N_1,N_2$ of $M$ are $I$-almost equal, i.e., $N_1\alm{=}N_2$, if $N_1\alm{\subseteq} N_2$ and $N_2\alm{\subseteq} N_1$. We will usually focus on ideals rather than submodules.

\subsection{Almost-pro-isomorphisms}
Here we briefly discuss the notion of almost mathematics in the world of pro-objects. See the detailed discussion in \cite[Section 11.3]{Bhatt2017notes}. We fix a ring $A$ and an $A$-flat ideal $I$ such that $I^2=I$. Let us consider a simpler setting: all objects are projective systems $\set{M_j}_{j\in J}$ of $A$-modules indexed by the positive integers.
\begin{defn}
  A pro-$A$-module $\set{M_j}_{j\in J}$ of $A$-modules is \emph{almost-pro-zero} if for each $w\in I$ and $j\in J$, there exists some $k\geqslant j$ such that the transition map $M_k\to M_j$ has its image killed by $w$; a map $\set{M_j}_{j\in J}\to\set{N_k}_{k\in K}$ of pro-$A$-modules is called an \emph{almost-pro-isomorphism} if the kernel and cokernel pro-objects are almost-pro-zero.
\end{defn}
In particular, we need the following lemma from \cite[Corollary 11.3.5]{Bhatt2017notes}.

% change R to rdf
\begin{lem}\label{m-lem-AlmostProIsomoprhism}
  Let $\set{M_j}_{j\in J}\to \set{N_k}_{k\in K}$ be an almost-pro-isomorphism, and let $F:\mrm{Mod}_R\to\mrm{Mod}_R$ be an $R$-linear functor. Then $\rdf\varprojlim\limits_j F(M_j)\to \rdf\varprojlim\limits_k F(N_k)$ is an almost isomorphism on each cohomology group.
\end{lem}

\subsection{Perfectoid algebras}\label{ssn-PerfectoidAlgebra}
We will freely use the language of perfectoid spaces \cite{Scholze2012}. Throughout this paper we will always work in the following situation: Let $A$ be a complete unramified regular local ring of mixed characteristic $p$ that has an $F$-finite residue field $k$. By Cohen's structure theorem $A\cong V\ps{\nseq{2}{d}}$ where $V$ is the coefficient ring of $A$. Let $k^{\pc}$ be the perfect closure of $k$ and $W(k^{\pc})$ be the Witt vectors of $k^{\pc}$. Let $A_0$ be the $p$-adic completion of $A\otimes_V W(k^{\pc})$.

Let $\kc$ be the $p$-adic completion of $W(k^{\pc})[\pir{p}]:=\cup_{i=1}^{\infty} W(k^{\pc})[p^{1/p^i}]$ and $K=\kc[\frac{1}{p}]$. Then $K$ is a perfectoid field with $\kc$ its valuation ring. An element $\pi\in \kc$ that satisfies $\abs{p}\leqslant\abs{\pi}< 1$ is called a \emph{pseudo-uniformizer}. All theorems we cite in this section work for any choice of pseudo-uniformizer (usually we choose $\pi=p$). Let $A_{\infty,0}$ be the $p$-adic completion of $A_0[\pir{p},\pir{x_2},\ldots,\pir{x_d}]$. Then $A_{\infty,0}$ is an integral perfectoid $\kc$-algebra, and $A_{\infty,0}[\frac{1}{p}]$ is a perfectoid $K=\kc [\frac{1}{p}]$-algebra.

\begin{rmk} 
  We note that $A_{\infty,0}$ is also referred to as a perfectoid $\kc$-algebra (without ``integral''). The difference is very technical and will not affect any conclusion in our proofs. Explicitly, a perfectoid $\kc$-algebra $\cA$ is a $\pi$-adically complete $\kc$-algebra flat over $\kc$, and the map $\cA/\pi^{1/p}\rightarrow \cA/\pi$ is an isomorphism. Let $\cA_{*}$ be the set of elements in $\cA[\frac{1}{\pi}]$ that are $(\pir{\pi})$-almost in $\cA$, i.e., $\cA_{*}=\set{a\in\cA[\frac{1}{\pi}]\vert \pir{\pi} a \in \cA}$. By definition we have $\cA\alm{\cong}\cA_{*}$ for any perfectoid $\kc$-algebra $\cA$. An integral perfectoid $\kc$-algebra $\cA$ is a perfectoid $\kc$-algebra such that $\cA\cong \cA_{*}$ (an honest isomorphism). Since we always work in the $\pir{\pi}$-almost world, this difference will not affect anything.
\end{rmk}

We next state a result of Andr\'e \cite[Section 2.5]{Andre2018d} in a form we need. See also \cite[Theorem 1.5]{Bhatt2017} or \cite[Theorem 9.4.3]{Bhatt2017notes}.
\begin{thm}\label{m-thm-pPowerRoots}
  Let $\vr{A}$ be an integral perfectoid $\kc$-algebra and let $\pi$ be a pseudo-uniformizer of $\kc$. Let $g\in \vr{A}$ be an element. Then there exists a map $\vr{A}\rightarrow \vr{B}$ of integral perfectoid $\kc$-algebras that is almost faithfully flat modulo $\pi$ such that the element $g$ admits a compatible system of $p$-power roots $g^{1/p^k}$ in $\vr{B}$.
\end{thm}

 We need this compatible system of $p$-power roots of $g$ to make use of the following remarkable result of Andr\'e, which is referred to as the ``Perfectoid Abhyankar Lemma'' \cite[Theorem 0.3.1]{Andre2018l}. Again, we rephrase it into a form that suits our objectives. Here, for any perfectoid $K$-algebra $\cA$, we use $\vr{A}$ to denote its ring of power-bounded elements, i.e., elements whose powers form a bounded subset in $\cA$. The ring $\vr{A}$ is a perfectoid $\kc$-algebra if $\cA$ is a perfectoid $K$-algebra.
\begin{thm}\label{m-thm-PerfectoidAbhyankar}
  Let $\vr{A}$ be a perfectoid $\kc$-algebra, and $\cA$ a perfectoid $K$-algebra. Suppose that $g\in \vr{A}$ is a nonzerodivisor that admits a compatible system of $p$-power roots of $g$. Let $\cB'$ be a finite \'etale $\cA[\frac{1}{g}]$-algebra. Then
  \begin{enumerate}
  \item There exists a larger perfectoid algebra $\cB$ between $\cA$ and $\cB'$ such that the inclusion $\cA\rightarrow \cB$ is continuous. We have $\cB[\frac{1}{g}]=\cB'$, and $\vr{B}$ is contained in the integral closure of $\vr{A}$ and this inclusion is a $\pir{(pg)}$-almost isomorphism.
  \item For any $m\in \NN$, $\vr{B}/p^m$ is $\pir{(pg)}$-almost finite \'etale over $\vr{A}/p^m$.
  \end{enumerate}
\end{thm}

Typically, one has a complete local domain $R$ module-finite over $A$. One often starts with $R$ and chooses $A$. We can choose $g\in A$ to be a discriminant of $R$ over $A$. Thus $R_g$ is finite \'etale over $A_g$. We apply \cref{m-thm-pPowerRoots} to $A_{\infty,0}$ and $g\in A$ to obtain an integral perfectoid $\kc$-algebra $A_{\infty}$ that contains a compatible system of $p$-power roots of $g$. Then $R\otimes A_\infty [\frac{1}{g}]$ is finite \'etale over $A_\infty[\frac{1}{g}]$. The way we use \cref{m-thm-PerfectoidAbhyankar} is by setting $\vr{A}=A_\infty, \cA=A_\infty[\frac{1}{p}],\cB'=R\otimes \cA_\infty[\frac{1}{pg}]$.

\section{\texorpdfstring{$p$}{p}-Colon-Capturing}\label{sn-p-ColonCapturing}
% Throughout this section, we mainly focus on complete local domains and always assume that they have an $F$-finite residue field.
Let $R$ be a $d$-dimensional complete local domain of mixed characteristic $p$. We will define the $p$-colon-capturing property (\cref{m-defn-pcc}) and then start to prove that $\epf$ satisfies this property (\cref{m-thm-pColonCap}).

Let us discuss the behavior of the $\epf$ closure in $R^+$. For any ideal $I\subseteq R$, the $\epf$ closure of $IR^+$ in $R^+$ is the set of elements
\[
  \left\{ u\in R^+\, |\, \exists c\in R-\set{0}, \quad c^{\varepsilon}u\in IR^++p^NR^+ ,\quad \forall N\in\NN, \varepsilon\in\QQ^+\right\}.
\]

\begin{rmk}
  One can also use $c\in R^+$ instead of $c\in R$, i.e.,
  \[
    \left(IR^+\right)^{\epf}=\left\{ u\in R^+\, |\, \exists c\in R^+-\set{0}, \quad c^{\varepsilon}u\in IR^++p^NR^+ ,\quad \forall N\in\NN, \varepsilon\in\QQ^+\right\}.
  \]
  Because if some $c$ works, since $c$ is integral over $R$, it has a nonzero multiple $cs\in R$, and $r=cs$ will work as well.
\end{rmk}

We have an easy observation.
\begin{lem}\label{m-lem-EpfInAbsIntClosure}
  Let $R$ be an integral domain of mixed characteristic $p$. Then for any ideal $I\subseteq R$, we have
  \[
    (IR^+)^{\epf} = \bigcup (IS)^{\epf} \quad \text{for all $S\subseteq R^+$ module-finite over $R$.}
  \]
\end{lem}
\begin{proof}
  The containment $\supseteq$ is obvious. For the converse direction, suppose $u\in (IR^+)^{\epf}$. Then by definition we have $c^{\varepsilon}u \in (I,p^N)R^+$ for some $c\in R$. Since $u$ is algebraic over $R$, there is some module-finite extension $S$ of $R$ such that $u\in S$, and then \cref{i-def-epf} implies that $u\in (IS)^{\epf}$.
\end{proof}

We give the definition of our key property, $p$-colon-capturing.
\begin{defn}\label{m-defn-pcc}
  Let $R$ be a $d$-dimensional complete local domain of mixed
  characteristic $p$. Let $\seq{n}$ be part of a system of parameters
  of $R$. We say that $\seq{n}$ satisfies \emph{$p$-colon-capturing}
  if there is some fixed positive integer $N_0$ such that for all
  integers $N\geqslant N_{0}$ we have
  \[
    \left(\seq{n-1},p^N\right):_{R^+}x_n\subseteq
    \left(\left(\seq{n-1}, p^{N-N_0}\right)R^+\right)^{\epf}.
  \]
\end{defn}

The main theorem we aim to prove in this section is stated below.
\begin{thm}\label{m-thm-pColonCap}
  Let $R$ be a complete local domain of mixed characteristic $p$ with an $F$-finite residue field. Then all systems of parameters in $R$ satisfy $p$-colon-capturing.
\end{thm}

In order to prove the main theorem, we need two lemmas.
\begin{lem}\label{m-lem-A-pColonCap}
  Let $A$ be a regular complete local domain of mixed characteristic $p$. Let $\seq{d}$ be a system of parameters in $A$. Since $A$ is noetherian, we may choose some $N_0$ such that $ (\seq{n})A:_A p^{\infty}=(\seq{n})A:_A p^{N_0}$. Let $T$ be a $\pir{(pg)}$-almost flat $A$-algebra. Then for all $N\geqslant N_{0}$ and $1\leqslant n\leqslant d$, we have
  \[
    \left( \seq{n-1},p^{N}\right)T:_{T} x_{n}\alm{\subseteq}
    \left(\seq{n-1},p^{N-N_{0}}\right)T.
  \] 
\end{lem}
\begin{proof}
  Since $T$ is $\pir{(pg)}$-almost flat over $A$,
  \begin{equation*}
    (\seq{n})T:_{T}p^{\infty}\alm{=}\left( (\seq{n})A:_Ap^{\infty} \right)T=\left(  (\seq{n})A:_Ap^{N_0}\right)T\alm{=}(\seq{n})T:_{T}p^{N_0}.
  \end{equation*}
  Let $N\geqslant N_{0}$ be some arbitrary integer and let $u$ be an arbitrary element in
  $\left(\seq{n-1},p^N\right)T:_{T}x_n$. We have 
  \begin{equation}\label{eq:m-pcc-1}
    ux_n - wp^N  \in (\seq{n-1})T.
  \end{equation}
  for some $w\in T$, which implies that $w \in \left( \seq{n}
  \right)T:_{T} p^N $. Note that since
  \[ \left( \seq{n} \right)T:_{T} p^N
    \subseteq \left( \seq{n} \right)T:_{T} p^\infty \alm{=} \left(
      \seq{n} \right)T:_{T} p^{N_0},
  \]
  we have
  \begin{equation}\label{eq:m-pcc-11}
    w \alm{\in} \left( \seq{n} \right)T:_{T} p^{N_0}.
  \end{equation}
  Let $\varepsilon$ be an arbitrary positive
  rational number. We rewrite \eqref{eq:m-pcc-1} as
  \begin{align}
    ux_n - \left( wp^{N_0} \right)p^{N-N_0}  &\in (\seq{n-1})T \\
    \Rightarrow \epg ux_n - \left(\epg wp^{N_0} \right)p^{N-N_0}  &\in (\seq{n-1})T \label{eq:m-pcc-2}
  \end{align}
  \eqref{eq:m-pcc-11} shows that for any $\varepsilon$, $\epg wp^{N_0} \in \left( \seq{n} \right)T $. So for each $\varepsilon>0$, there is some $v_{\varepsilon}\in T$ such that $\epg wp^{N_0} - v_{\varepsilon}x_n\in \left(
    \seq{n-1} \right)T$. Combining this with \eqref{eq:m-pcc-2}, we have
  \begin{align*}
    \epg ux_n - (v_{\varepsilon}x_n)p^{N-N_0} &\in (\seq{n-1})T \\
    \Rightarrow \epg u - v_{\varepsilon}p^{N-N_0} &\in (\seq{n-1})T:_{T}x_n.
  \end{align*}
  Since $T$ is $\pir{(pg)}$-almost flat over $A$, we have $ \left(
    (\seq{n-1})T:_{T} x_n \right)\alm{\subseteq} \left( \seq{n-1}
  \right) T $. So we have
  \[
    \forall \varepsilon,\quad \epg u \alm{\in} \left(\seq{n-1},p^{N-N_0}\right)T.
  \]
  Since this is true for all $\varepsilon$, we conclude that $u\alm{\in}  \left(\seq{n-1},p^{N-N_0}\right)T$.
\end{proof}

\begin{lem}\label{m-lem-R-pColonCap}
  Let $R$ be a complete local domain that is module-finite over a regular
  complete local domain $A$, where both $R$ and $A$ are of mixed characteristic $p$.  Let $\seq[y]{n}$ be part of some
  system of parameters in $R$. 
  There exists some positive integer $N_0$ such that for any integer
  $N\geqslant N_{0}$ and any $R$-algebra $T$ that is $\pir{(pg)}$-almost
  flat over $A$,
  \[
    \left( \seq[y]{n-1},p^N \right)T:_T y_n  \alm{\subseteq} \left(
      \seq[y]{n-1},p^{N-N_0} \right)T.
  \]
\end{lem}
\begin{proof}
  Let $k$ be the number of elements in $\left\{ \seq[y]{n} \right\}$
  that are in $A$. We prove the lemma by induction on $n$ and
  $n-k$. The base cases $n=k$ for all $1\leqslant n\leqslant d$ follow from \cref{m-lem-A-pColonCap} with the same $N_0$.
  Let us choose the same $N_0$ from \cref{m-lem-A-pColonCap}. We assume that
  $n-k>0$. To simplify the notation, we write $\underline{y}$ for the sequence $y_2,\ldots,y_{n-1}$.

  If $y_n\not\in A$, then we can choose some $w_n \in (y_1,\underline{y},y_n)R\cap A$ that is not contained in any minimal prime of
  $\left( y_1,\underline{y} \right)R$ in $R$, and then $y_1,\underline{y},w_n$
  continue to be part of a system of parameters in $R$.  There is one more element of $y_1,\underline{y},w_n$ in $A$ than of
  $y_1,\underline{y},y_n$. By the induction hypothesis on $n-k$, there is some
  $N_{0}$ such that for any $N\geqslant N_{0}$ and any $R$-algebra $T$ that is $\pir{(pg)}$-almost
  flat over $A$,
  \begin{equation}\label{eq:m-pcc-3}
    \left( y_1,\underline{y},p^N \right)T:_{T}w_n \alm{\subseteq} \left( y_1,\underline{y},p^{N-N_0} \right)T.
  \end{equation}
  
  Since $T$ is an $R$-algebra, we have
  \begin{equation}\label{eq:m-pcc-4}
    \left( y_1,\underline{y},p^N \right)T:_{T} y_n = \left( y_1,\underline{y},p^N \right)T:_{T} ( y_1,\underline{y},y_n)\subseteq\left( y_1,\underline{y},p^N \right)T:_{T} (y_1,\underline{y},w_n)= \left( y_1,\underline{y},p^N \right)T:_{T}w_n.
  \end{equation}
  The result now follows from \eqref{eq:m-pcc-3} and \eqref{eq:m-pcc-4}.  

  In the remaining case, we can assume that $y_n\in A$. Without loss
  of generality, we assume that $y_1\not\in A$. By applying what we have proved
  above to the sequence $\underline{y},y_{n},y_{1}$, we know that
  there is some $N_{1}$ such that for all $N\geqslant N_{1}$ and all $R$-algebras $T$ that are $\pir{(pg)}$-almost
  flat over $A$,
  \begin{equation}\label{eq:m-pcc-5}
    \left(\underline{y},y_{n},p^{N}\right)T:_{T}y_{1}\alm{\subseteq} \left( \underline{y},y_n,p^{N-N_1} \right)T.
  \end{equation}
  Also, by applying the induction hypothesis on $n$ to the shorter sequence $\underline{y},y_n$, there is some $N_{2}$ such
  that for all $N\geqslant N_{2}$ and all $R$-algebras $T$ that are $\pir{(pg)}$-almost
  flat over $A$,
  \begin{equation} \label{eq:m-pcc-6}
    \left(\underline{y},p^{N}\right)T:_{T}y_{n}\alm{\subseteq} \left( \underline{y},p^{N-N_2}\right)T.
  \end{equation}
  
  Let $N\geqslant N_{1}+N_{2}$ be an integer. For any $u\in  \left(
    y_1,\underline{y},p^N \right)T:_{T} y_n$, we can write
  \begin{equation}\label{eq:m-pcc-7}\\
    y_nu=u_1y_1+\cdots+u_{n-1}y_{n-1}+vp^N
  \end{equation}
  for some $\seq[u]{n-1},v\in T$. Then $u_1y_1 \in \left(\underline{y},y_n,p^N\right)T$, which by \eqref{eq:m-pcc-5} implies
  that
  \[
    u_1  \alm{\in} \left( \underline{y},y_n,p^{N-N_{1}} \right)T.
  \]

  For any positive rational number $\varepsilon$, we have $\epg u_1 =
  b_2y_2+\cdots +b_ny_n+wp^{N-N_1}$ for some $b_2,\ldots,b_n,w\in T$ which depend on $\varepsilon$. Multiply \eqref{eq:m-pcc-7} by $\epg$ and
  make use of the expression of $\epg u_1$. This yields
  \begin{align*}    
    \epg y_nu &= \epg u_1y_1 +\cdots+\epg u_{n-1}y_{n-1}+\epg vp^N \\
    \Rightarrow    \epg y_n u - b_ny_1y_n &\in \left( \underline{y},p^{N-N_1} \right)T \\
    \Rightarrow    \epg u - b_ny_1 &\in \left(\underline{y},p^{N-N_1} \right)T:_{T}y_n \\
    \Rightarrow \epg u - b_ny_1 &\alm{\in} \left( \underline{y},p^{N-N_1-N_2} \right)T\quad    \text{(by \eqref{eq:m-pcc-6})}\\
    \Rightarrow   \epg u &\alm{\in} \left( y_1,\underline{y},p^{N-N_1-N_2} \right)T,
  \end{align*}
  and this is true for any positive rational number $\varepsilon$ and any $u\in  \left( y_1,\underline{y},p^N \right)T:_{T} y_n$. Hence, we can conclude that
  \[
    \left( \seq[y]{n-1},p^N \right)T:_{T} y_n\alm{\subseteq} \left(
      y_1,\ldots,y_{n-1},p^{N-N_1-N_2} \right)T.
  \]
\end{proof}

Next we discuss some perfectoid constructions that will be used in the proof of \cref{m-thm-pColonCap}.
\begin{cons}\label{m-cons-AndrePerfectoidConstruction}
  Let $A\to S$, a module-finite map of complete local domains of mixed characteristic $p$, be given where $A$ is regular and unramified with an $F$-finite residue field. We apply the same construction as in \cref{ssn-PerfectoidAlgebra} to obtain $A_{\infty,0}$. Since $A\rightarrow S$ is generically \'etale, there is some element $g$ in $A$ such that $(p,g)$ generates a height 2 ideal and $A_{pg}\rightarrow S_{pg}$ is finite \'etale. Let $A_{\infty}$ be obtained by applying \cref{m-thm-pPowerRoots} to $A_{\infty,0}$ and $g$. Then by \cref{m-thm-PerfectoidAbhyankar} (where $\cA=A_{\infty}[\frac{1}{p}]$ and $\cB'=A_{\infty}[\frac{1}{pg}]\otimes_A S$), we are able to find an $S$-algebra $\mbc$ satisfying the following properties:
  \begin{itemize}
  \item $\mbc$ is $\pir{(pg)}$-almost flat over $A$.
  \item There exists a $\pir{(pg)}$-almost map from $\mbc$ to $S^{pg}$
    where $S^{pg}$ is the integral closure of $\cp{S}$ in
    $\cp{S}[\frac{1}{pg}]$, and $\cp{S}$ is the $p$-adic completion of
    $S^+$.
  \end{itemize} 
  For proofs, see \cite[Lemma 3.8, Lemma 3.9]{Heitmann2018}. 
\end{cons}

A direct consequence of the construction above is the following lemma.
\begin{lem}\label{m-lem-EpfSufficientCondition}
  With notation as above, let $I\subseteq S $ be an ideal of $S$ and $u\in S$. If $u$ is $\pir{(pg)}$-almost in $I\mbc$, then $u\in I^{\epf}$.
\end{lem}
\begin{proof}
  Since $\mbc$ maps $\pir{(pg)}$-almostly to $S^{pg}$, we have
  $u\alm{\in} IS^{pg}$. By \cite[Lemma 3.3]{Heitmann2018} we know that
  $S^{pg}$ is $\pir{(pg)}$-almost isomorphic to $\cp{S}$. Hence
  $u\alm{\in} I\cp{S}$. Then \cite[Lemma 3.1]{Heitmann2018} finishes
  the proof.
\end{proof}

We are ready to prove our main result of this section.
\begin{proof}[Proof of \cref{m-thm-pColonCap}]
  Since $R$ is a complete local domain, by Cohen's structure theorem, there is a complete regular local domain $A$ such that $A\to R$ is a module-finite extension. So $A$ has an $F$-finite residue field. We fix this $A$ for the remainder of the proof.
  
  For any $\seq{n}$ that is part of a system of parameters, we want to
  prove that there is some $N_{0}$ such that for all $N\geqslant N_{0}$,
  \[
    \left( \seq[y]{n-1},p^N \right)R^+:_{R^+} y_n \subseteq \left(\left(
        \seq[y]{n-1},p^{N-N_0} \right)R^+\right)^{\epf}.
  \]
  We apply \cref{m-lem-R-pColonCap} to the system of parameters $\seq[y]{n}$. We learn that there is some positive integer $N_{0}$ such that for any $N\geqslant N_{0}$ and any $R$-algebra $T$ that is $\pir{(pg)}$-almost flat over $A$,
  \begin{equation}\label{eq:m-pcc-8}
    \left( \seq[y]{n-1},p^N \right)T:_T y_n  \alm{\subseteq} \left(\seq[y]{n-1},p^{N-N_0} \right)T.
  \end{equation}

  Let $u$ be an arbitrary element in $\left( \seq[y]{n-1},p^N \right)R^+:_{R^+}y_n$. Then we have
  \begin{equation}\label{eq:m-pcc-9}
    y_nu = y_1u_1+\cdots+y_{n-1}u_{n-1}+vp^N
  \end{equation}
  for some $\seq[u]{n-1},v\in R^+$. All elements here are integral over $R$. Hence, there is some module-finite extension $S$ of $R$ such that this relation holds in $S$, i.e., $u\in \left(\seq[y]{n-1},p^N \right)S:_S y_n$. Applying \cref{m-cons-AndrePerfectoidConstruction} to $A\to S$, we obtain an $S$-algebra $\mbc$ that is $\pir{(pg)}$-almost flat over $A$. $\mbc$ is also an $R$-algebra. So we set $T=\mbc$ in \eqref{eq:m-pcc-8} and obtain that
  \begin{equation}\label{eq:m-pcc-10}
    \left( \seq[y]{n-1},p^{N}\right)\mbc:_{\mbc} y_{n} \alm{\subseteq} \left(
      \seq[y]{n-1},p^{N-N_{0}}\right)\mbc
  \end{equation}
  for all $N\geqslant N_{0}$. Since the relation \eqref{eq:m-pcc-9} maps to a relation in $\mbc$, we see that $u$ is in the left-hand side of \eqref{eq:m-pcc-10}. Hence it is $\pir{(pg)}$-almost in the right-hand side of \eqref{eq:m-pcc-10}. Then, by \cref{m-lem-EpfSufficientCondition}, we know that $u\in \left(\left(\seq{n-1},p^{N-N_0}\right)S\right)^{\epf}$ for all $N\geqslant N_0$. This completes the proof, by \cref{m-lem-EpfInAbsIntClosure}.
\end{proof}

% end of section{$p$-Colon-Capturing}

\section{Weak \texorpdfstring{$\epf$}{epf} Closure}\label{sn-WeakEPFClosure}
In this section, we develop a new closure operation, called ``weak $\epf$ closure'', and denote it by $\wepf$. We prove that it satisfies not only the generalized colon-capturing property (\cref{m-thm-WepfGCC}), but also two stronger colon-capturing properties (\cref{m-prop-WepfStrongColonCap}). We also show that $\wepf$ is a Dietz closure satisfying the algebra axiom (\cref{app-thm-WEPFDietz}). Let us begin with the definition of $\wepf$.

\begin{defn}\label{m-defn-WEPF}
  Let $R$ be a complete local domain of mixed characteristic $p>0$. Let $I\subseteq R$ be an ideal. Then the \emph{weak $\epf$ closure} of $I$, denoted by $I^{\wepf}$, is defined to be $I^{\wepf}:=\bigcap_{N=1}^{\infty} \left( I,p^N \right)^{\epf}$. Similarly for finitely generated $R$-modules $W\subseteq M$, we define $W^{\wepf}:=\bigcap_{N=1}^{\infty} \left( W+p^N M \right)^{\epf}_M$.
\end{defn}
\begin{rmk}
  It is clear from the definition that $I^{\epf}\subseteq I^{\wepf}$. So the weak $\epf$ closure also satisfies the usual colon-capturing. It is
  not hard to see that $I^{\wepf}\subseteq I^{\rof}$: Let $u\in
  I^{\wepf}$. For any rank 1 valuation $\nu$ on $R^+$, any $N\in \NN$
  and any $\varepsilon\in\QQ^+$, since $u\in I^{\wepf}\subseteq
  (I,p^N)^{\epf}$, there is some $c\in R$ such that $c^{\delta} u \in
  (I,p^N)R^+ $ for any $\delta\in \QQ^+$. We choose $\delta$ small
  enough such that $\nu(c^{\delta})< \varepsilon$. The existence of such a sequence of $c^{\delta}$ implies that $u\in I^{\rof}$.
\end{rmk}

Since the $\epf$ closure on complete regular local domains with $F$-finite residue fields is trivial by \cite[Theorem 3.9]{Heitmann2018}, we have 
\[
  I^{\wepf}=\bigcap_{N=1}^{\infty} \left(I,p^N\right)^{\epf}  =\bigcap_{N=1}^{\infty} \left(I,p^N\right) = I.
\]
Therefore we have
\begin{cor}\label{m-cor-WEPFTrivial}
The $\wepf$ closure is trivial on complete regular local domains of mixed characteristic with $F$-finite residue fields.
\end{cor}

Next we want to show that the $\wepf$ closure satisfies the generalized colon-capturing axiom using $p$-colon-capturing.

\begin{thm}\label{m-thm-WepfGCC}
  Let $R$ be a complete local domain of mixed characteristic $p$ with an $F$-finite residue field. Then the $\wepf$ closure on $R$ satisfies the generalized colon-capturing axiom.
\end{thm}
\begin{proof}
  Let $M$ be an $R$-module and let $\seq{n+1}$ be a partial system of parameters of $R$. Let $f:M\rightarrow R/I$ be a morphism of $R$-modules where $I=(\seq{n})R$ and $f(v)=\bar{x}_{n+1}$in $R/I$. Suppose that $u\in (Rv)^{\wepf}_M\cap \aker f$. We want to show that $u\in (Iv)^{\wepf}_M$.

  If we apply $p$-colon-capturing (\cref{m-thm-pColonCap}) to the system of parameters $\seq{n+1}$, we learn that there is some $N_{0}$ such that for any $N\geqslant N_{0}$ we have
  \begin{equation}\label{eq:m-wec-1}
    \left( \seq{n},p^{N}\right)R^+:_{R^{+}} x_{n+1} \subseteq \left(\left( \seq{n},p^{N-N_{0}}\right)R^{+}\right)^{\epf}.
  \end{equation}

  Since $u\in (Rv)^{\wepf}_M$, we have $u\in \bigcap_{N=1}^{\infty} (Rv+p^NM)^{\epf}_M$, i.e., $u\in (Rv+p^NM)^{\epf}_M$ for all positive integers $N$ or equivalently all $N\geqslant N_{0}$. Fix $N=N_1\geqslant N_{0}$. We know that there is some nonzero element $c\in R$ such that for any $\varepsilon\in\QQ^+$, we have $c^{\varepsilon}\otimes u\in \aim(R^+\otimes_{R}v\rightarrow M^+)+p^{N_1}M^+$, where $M^+= R^+\otimes_{R}M$ as in \cref{i-ntn-plus}. So there is some $a\in R^+$ and $\mu\in M^{+}$ such that $c^{\varepsilon}\otimes u=a\otimes v+p^{N_1}\mu$. We apply $1\otimes_{R}f$ and note that $u\in \aker f$. Hence, $ a x_{n+1}+p^{N_1} (1\otimes_{R} f)(\mu)\in IR^+$, which gives us $a\in (\seq{n},p^{N_1})R^+:_{R^+} x_{n+1}$. By \eqref{eq:m-wec-1}, we have $a\in ((\seq{n},p^{N_1-N_0})R^+)^{\epf}$. Hence, there is some $c'$ (depending on $N_1-N_0$) such that
  \[
    (c')^{\varepsilon}a\in (\seq{n})R^++p^{N_1-N_0}R^+.
  \]
  Now everything on the right-hand side of $\left(c'c\right)^{\varepsilon}\otimes u=(c')^{\varepsilon}a\otimes v+(c')^{\varepsilon}p^{N_1}\mu$ is
  in $\aim(IR^+\otimes_{R}v\to M^+)+p^{N_1-N_0}M^+$. We have
  \[
    \left( c'c\right)^{\varepsilon}\otimes u \in
    \aim(IR^+\otimes_{R}v\rightarrow M^+)+p^{N_1-N_0}M^+.
  \]
  Therefore, $u\in (Iv,p^{N_1-N_0})^{\epf}_M$ for all
  $N_1\geqslant N_{0}$. We conclude that $u\in (Iv)^{\wepf}_M$.
\end{proof}

\begin{rmk}
  In the proof above, the element $c'$ depends on $N_1-N_0$. Hence, we can not use the same $c'$ when $N_1$ changes. Therefore, we do not have an obvious way to prove that $\epf$ closure satisfies the generalized colon-capturing axiom. So far as we know, this is an open question. 
\end{rmk}

We also prove that the the $\wepf$ closure satisfies some strong colon-capturing conditions (both versions A and B) defined in \cite[Definition 3.9]{R.G.2016}.

\begin{prop}\label{m-prop-WepfStrongColonCap}
  %The $\wepf$ closure satisfies the following strong colon-capturing conditions: 
  Let $\seq{k}$ be a partial system of parameters in a
  complete local ring $R$ of mixed characteristic $p$ with an $F$-finite residue field. Let $t,a$ be two positive integers. Then
  \begin{enumerate}
  \item $\left((x_1^t,x_2,\ldots,x_k)R\right)^{\wepf}:_R x_1^a\subseteq \left((x_1^{t-a},x_2,\ldots,x_k)R\right)^{\wepf}$ for all $a<t$;
  \item $\left((\seq{k})R\right)^{\wepf}:_R x_{k+1}\subseteq \left((\seq{k})R\right)^{\wepf}$.
  \end{enumerate}
\end{prop}
\begin{proof}
  (1) For the first containment, consider an element $u\in
  \left((x_1^t,x_2,\ldots,x_k)R\right)^{\wepf}:_R x_1^a$. Then 
  \[
    ux_1^a\in\left( (x_1^t,x_2,\ldots,x_k)R\right)^{\wepf}.
  \]
  For any $N$ there is some $c_{N}\in R$ such that for any $\varepsilon$ we
  have $ c_{N}^{\varepsilon} ux_{1}^{a} \in \left(
    x_{1}^{t},x_{2},\ldots,x_{k},p^{N}\right)R^{+} $.
  So there is some $v\in R^+$ such that $c^{\varepsilon}_{N}x_1^au - x_1^tv \in \left(x_2,\ldots,x_k,p^N\right)R^+$. Then we have
  $$ c^{\varepsilon}_{N}u-x_1^{t-a}v \in (x_2,\ldots,x_k,p^N)R^+:_{R^+}x_1^a.$$
  By \cref{m-thm-pColonCap} there is some $N_0$ such that for any $N\geqslant
  N_{0}$ we have $ c^{\varepsilon}_{N}u-x_1^{t-a}v\in \left(
    (x_2,\ldots,x_k,p^{N-N_0})R^+ \right)^{\epf}$. So there is another
  element $d_{N-N_0}\in R^+$ such that for any  positive rational number
  $\delta$, we have
  \begin{align*}
    d_{N-N_0}^{\delta} (c^{\varepsilon}_{N}u-x_1^{t-a}v) &\in \left(x_2,\ldots,x_k,p^{N-N_0}\right)R^+  \\
    \Rightarrow  d_{N-N_0}^{\delta} c^{\varepsilon}_{N}u &\in \left(x_1^{t-a},x_2,\ldots,x_k,p^{N-N_0}\right)R^+.
  \end{align*}
  We can choose $\delta=\varepsilon$. Hence, we conclude that $u \in \left(( x_1^{t-a},x_2,\ldots,x_k,p^{N-N_{0}})R \right)^{\epf}$. Since this is true for all $N\geqslant N_{0}$, we conclude that $u\in \left((x_{1}^{t-a},x_{2},\ldots, x_{k})R\right)^{\wepf}$.

  (2) For the second containment, the proof is similar. For any $u\in \left(  (x_1,x_2,\ldots,x_k)R\right)^{\wepf}:_R x_{k+1}$, we have $ux_{k+1}\in \left((x_1,x_2,\ldots,x_k)R\right)^{\wepf}$. So for any $N$ there is some $c_{N}$ 
  such that for all $\varepsilon$ we have
  \[
    c^{\varepsilon}_{N}x_{k+1}u \in (x_1,x_2,\ldots,x_k,p^N)R^+.
  \]
  So we have $c^{\varepsilon}_{N}u \in
  (x_1,x_2,\ldots,x_k,p^N)R^+:_{R^+}x_{k+1}$. Again, by \cref{m-thm-pColonCap}, we have some $N_0$ such that for all $N\geqslant N_{0}$, $c^{\varepsilon}_{N}u \in
  \left( (x_1,x_2,\ldots,x_k,p^{N-N_0})R^+ \right)^{\epf}$.

  So there is another element $d_{N-N_0}$ such that for any positive rational $\delta$, we have
  \[
    d_{N-N_0}^{\delta} c^{\varepsilon}_{N}u \in \left(x_1,x_2,\ldots,x_k,p^{N-N_0}\right)R^+.
  \]
  For the same reason as in the end of the proof of (1), we conclude that $u\in \left( (x_1,x_2,\ldots,x_k)R \right)^{\wepf}$.
\end{proof}

\begin{rmk}
  We point out that the same arguments in both \cref{m-thm-WepfGCC} and \cref{m-prop-WepfStrongColonCap} work for the $\rof$
  closure. Interested readers can work out the details of the proof.
\end{rmk}

Next, we prove that $\wepf$ satisfies all of Dietz's and R.G.'s axioms (\cref{app-axm-DietzAxioms}). As mentioned in \cref{sn-introduction}, this gives a new proof of the existence of big Cohen-Macaulay algebras. The results in this section are not used in \cref{sn-PhantomExtensions} and \cref{sn-p-PositiveCharCase}.

In \cite[Proposition 7.2]{R.G.2016} R.G. proved that the usual $\epf$
closure satisfies the first six axioms. Next we prove
\begin{thm}\label{app-thm-WEPFDietz}
  The $\wepf$ closure satisfies all axioms in \cref{app-axm-DietzAxioms}.
\end{thm}
\begin{proof}
  (1) Since $\left(Q+p^{N}M\right)^{\epf}_{M}$ is a submodule
  containing $Q$ for each
  positive integer $N$, we conclude that the intersection, that is the
  $\wepf$ closure of $Q$, is a submodule of $M$ containing $Q$.

  (2) Since the ambient module is always $M$ here, we omit the subscript
  and write $Q^{\wepf}$ for $Q_M^{\wepf}$.
  We need to show that $Q^{\wepf}$ is a $\wepf$ closed submodule, i.e.,
  $(Q^{\wepf})^{\wepf}=Q^{\wepf}$. 
  Let $u\in (Q^{\wepf})^{\wepf}$.
  Then, by definition, for each $N$ there is some $c_{N}\in R$ such that
  for any $\varepsilon\in\QQ^{+}$  we have
  \[
    c_{N}^{\varepsilon} \otimes u \in \aim((Q^{\wepf})^{+}\to M^+) +
    p^{N}M^+.
  \]
  So there exist some elements $r_{1},
  \ldots,r_{n}\in R^{+}$, $q_{1},\ldots,q_{n}\in
  Q^{\wepf}$, and $v\in M^{+}$ such that
  \[
    c_{N}^{\varepsilon} \otimes u = \sum_{i=1}^{n} r_{i}\otimes q_{i} +
    p^{N} v.
  \] 
  Look at one $q_{i}$. For each positive integer $N_{i}$, there is
  some $c_{i,N_{i}}\in R^+$ such that for any $\varepsilon_{i}\in\QQ^{+}$ we
  have $c_{i,N_{i}}^{\varepsilon_{i}} q_{i}\in \aim(Q^+\to M^+)+p^{N_i}M^+$. Choose $N_{i}$ to be $N$, and we have
  \[
    (\prod_{i=1}^{n} c_{i,N}^{\varepsilon_{i}})c_{N}^{\varepsilon}\otimes u \in \aim (Q^+\to M^+)+p^N M^+.
  \] 
  Choose $\varepsilon_{i}$ to be $\varepsilon$, this implies that
  $u\in (Q+p^{N}M)^{\epf}$. Since this is true for any $N$, we
  conclude that $u\in Q^{\wepf}$.

  (3) This is true for $\epf$ closure, and hence we have
  $(Q+p^{N}W)_{W}^{\epf}\subseteq (M+p^{N}W)_{W}^{\epf}$ for all
  positive integer $N$. Hence 
  $\bigcap_{N}(Q+p^{N}W)_{W}^{\epf}\subseteq
  \bigcap_{N}(M+p^{N}W)_{W}^{\epf}$, i.e., $Q^{\wepf}_{W}\subseteq
  M^{\wepf}_{W}$.

  (4) Note that
  \[
    f(Q_{M}^{\wepf}) = f \left( \bigcap_{N} (Q+p^{N}M)_{M}^{\epf}\right)
    \subseteq \bigcap_{N} f\left( (Q+p^{N}M)_{M}^{\epf}\right).
  \]
  For each term we have
  \[
    f\left( (Q+p^{N}M)_{M}^{\epf}\right)\subseteq \left(
      (f(Q)+p^{N}f(M))_{W}^{\epf}\right) \subseteq \left((f(Q)+p^{N}W)_{W}^{\epf}\right).
  \]
  We conclude that
  \[
    f(Q_{M}^{\wepf}) 
    \subseteq \bigcap_{N} \left(
      (f(Q)+p^{N}W)_{W}^{\epf}\right)=\left( f(Q)\right)_{W}^{\wepf}.
  \] 

  (5) Again, we omit the subscript as the ambient module is always $M$.
  Assume that $Q$ is $\wepf$-closed. Let $ \overline{u}$ be an element
  in $0_{M/Q}^{\wepf}$. We want to show that $ \overline{u}=0$.  For each
  $N$ we have $c_{N}^{\varepsilon} \overline{u} \in p^{N} (M/Q)^{+}$. Since we have
  $(M/Q)^{+}\simeq M^{+}/\aim(Q^{+}\to M^{+})$, we conclude that
  $c_{N}^{\varepsilon} u \in \aim(Q^{+}\to M^{+}) +p^{N}M^{+}$ for any $u$
  that is a preimage of $ \overline{u}$ in $M$. Therefore, $u\in
  Q_{M}^{\wepf}=Q$. Hence $ \overline{u}=0$ in $M/Q$.

  (6) Note that $R$ is of mixed characteristic $p$. So $p\in \mm$ and hence $\mm+p^{N}R=\mm \Rightarrow
  \mm^{\wepf}=\mm^{\epf}=\mm$. For $0^{\wepf}=0$, we prove by citing known
  results. The same argument in the last part of the proof in
  \cite[Proposition 7.2]{R.G.2016} works directly for $\wepf$. This
  argument also works for $\rof$ closure, i.e., $0^{\rof}=0$, which
  implies that $0^{\wepf}=0$. Dietz pointed out that $0^{\cl}=0$ follows from the other 5 axioms and generalized colon-capturing in \cite[Lemma 1.3(e)]{Dietz2018}. Thus, for the case where the complete local
  ring has a $F$-finite residue field, we have an alternate
  prove of $0^{\wepf}=0$ using the generalized colon-capturing property \cref{m-thm-WepfGCC}.

  (7) See \cref{m-thm-WepfGCC}.

  (8) We also point out that similar arguments to those in \cite[Proposition
  3.19]{R.G.2018} work for $\wepf$ closure and therefore, $\wepf$ also
  satisfies the algebra axiom.
\end{proof}

\begin{rmk}\label{m-rmk-ClIff}
  Note that if a closure operation satisfies both the functoriality axiom and the semi-residuality axiom in \cref{app-axm-DietzAxioms}, then the statement in semi-residuality axiom can be improved to $Q^{\cl}_M=Q$ if and only if $0^{\cl}_{M/Q}=0$. The ``only if'' direction is the semi-residuality axiom. The ``if'' direction comes from the functoriality axiom: consider the map $f:M\to M/Q$, we have $f(Q^{\cl}_M)\subseteq \left( f(Q) \right)_{M/Q}^{\cl}=0^{\cl}_{M/Q}=0$. So $f(Q^{\cl}_M)\subseteq\aker(f)=Q$.
\end{rmk}

The following proposition is proved by using standard techniques of reducing closure problem for submodules to the case of ideals. The proof we include here is basically the same as the proof of \cite[Proposition 8.7]{Hochster1990}. 

\begin{prop}
  Let $(R,\mm)$ be a complete regular local domain of mixed characteristic with $F$-finite residue fields. Then every submodule $W$ of a finitely generated module $M$ is $\wepf$-closed.
\end{prop}
\begin{proof}
We want to show that for any $u\in M$ not in $W$, we have $ u\nin W^\wepf $. We may replace $W$ by a submodule of $M$ maximal with respect to not containing $u$, and we may replace $M, W$, and $u$ by $M/W,0$, and $u+W$. Then $u$ is in every nonzero submodule of $M$. We claim that $M$ is now of finite length, i.e., $M$ has only one associated prime $\mm$.

Suppose that $M$ has two associated primes $P,Q$. Let $v_1,v_2\in M$ be elements such that $\ann_R(v_1)=P$ and $\ann_R(v_2)=Q$. Then $Rv_1\subseteq M$ is isomorphic to $R/P$. So every element in $Rv_1$ has annihilator $P$. Similarly, every element in $Rv_2$ has annihilator $Q$. Then $P=Q$ as $u$ is in both submodules. Thus $\ass(M)$ consists of only one prime $P$.

Next we show that $P$ must be the maximal ideal $\mm$. If not, then the image of $(\mm^n+P)/P$ in $R/P\hookrightarrow M$ contains $u$ for every positive integer $n$. But $\bigcap_{n=1}^\infty (\mm^n+P) = P$ as $R$ is noetherian. This is impossible, and we have $\ass(M)=\set{\mm}$.

Since $u$ is in every nonzero submodule of $M$, $u$ spans the socle in $M$, and $M$ is an essential extension of $K=R/\mm\cong Ru$. Since R is regular, there exists an irreducible $\mm$-primary ideal $J\subseteq \ann_R(M)$. The Artin ring $R/J$ is self-injective, and $M$ is an essential extension of $K$ as an $R/J$-module. It follows that $M$ can be embedded in $R/J$. It will then suffice to show that $0$ is $\wepf$ closed in $R/J$, i.e., that $J$ is $\wepf$ closed in $R$ by \cref{m-rmk-ClIff}. Then \cref{m-cor-WEPFTrivial} finishes the proof.
\end{proof}

\begin{cor} 
   Let $R$ be a complete regular local domain of mixed characteristic with $F$-finite residue field. Then every submodule $W$ of a finitely generated module $M$ is $\epf$-closed.
\end{cor}
\begin{proof}
  Since $W^\epf\subseteq W^\wepf=W$.
\end{proof}

% end of section "weak epf closure"

\section{Phantom Extensions}\label{sn-PhantomExtensions}
In \cite[Theorem 5.13]{Hochster1994a}, Mel Hochster and Craig Huneke proved that in characteristic $p$, all module-finite extensions of complete local rings are phantom in the tight closure sense. We prove a similar result, namely any module-finite extension of a complete local ring of mixed characteristic $p$ with an $F$-finite residue field is $\epf$-phantom (hence $\wepf$- and $\rof$-phantom).

\begin{thm}\label{m-thm-EPFPhantom}
  If $R\rightarrow S$ is a module-finite extension of complete local domains of mixed characteristic $p$ with an $F$-finite residue field, then this map is $\epf$-phantom.
\end{thm}

% We first note following lemma

% \begin{lem}\label{m-lem-VecEpfSufficientCondition}
%   Let $R$ be an integral domain of mixed characteristic $p$. Let $\cp{R}$ be the $p$-adic completion of $R^+$. Suppose that $\vec{x}= \begin{pmatrix} x_1 , \ldots , x_n \end{pmatrix}^{\vee},\vec{a}_1=\begin{pmatrix} a_{11} , \ldots , a_{1n} \end{pmatrix}^{\vee},\ldots,\vec{a}_k=\begin{pmatrix} a_{k1} , \ldots , a_{kn} \end{pmatrix}^{\vee} $ are column vectors in $\dsm[R]{n}$. If there is some $c\in R$ such that for any $\varepsilon\in\QQ^+$, $ c^{\varepsilon}\vec{x} \in (\vec{a}_1,\ldots,\vec{a}_k)\dsm[\cp{R}]{n}$, then $\vec{x}\in \left(\left( \vec{a}_1,\ldots,\vec{a}_k \right)\dsm[R]{n} \right)^{\epf}$.
% \end{lem}
% This is the same result as \cite[Lemma 3.1]{Heitmann2018} for free modules of finite rank. The same proof works here.

Let us discuss the definition of phantom extension and introduce some notions. See also Dietz \cite[Discussion 2.4]{Dietz2010}.

Suppose that $(R,\mm)$ is a complete local domain. Let $S$ be a module-finite extension of $R$. Then $S/R$ is a finitely generated module over $R$. So it has a minimal $R$ basis $\seq[e]{n_0}$. The set of column vectors of the $n_0\times n_0$ identity matrix form an $R$ basis for $\dsm{n_0}$, which we denote by $\seq[f]{n_0}$. We can map a free module $\dsm{n_0}$ onto $S/R$ by $f_i\mapsto e_i$. This map has a finitely generated kernel. Suppose that it is minimally generated by $n_1$ elements. Then we have a minimal resolution of $S/R$:
\[
  \dsm{n_1}\rightarrow \dsm{n_0}\rightarrow S/R.
\]
Suppose that the map $\dsm{n_1}\rightarrow \dsm{n_0}$ is represented by a $n_0\times n_1$ matrix $\nu$ with all entries in $\mm$. Comparing this resolution with the original exact sequence $0\rightarrow R\rightarrow S\rightarrow S/R\rightarrow 0$, we have the following commutative diagram:
\[
  \xymatrix{
    0\ar[r] & R\ar[r] & S\ar[r] & S/R\ar[r] & 0 \\
    & \dsm{n_1}\ar[r]_{\nu}\ar[u]^{\phi} & \dsm{n_0}\ar[r]\ar[u] & S/R\ar[r]\ar[u] & 0
  },
\]
where $\phi$ is a $1\times n_1$ matrix with entries in $\mm$. Applying $\ahom_R(-,R)$ to the left square, and using the identification $\ahom_R(R^{\oplus l},R)\cong R^{\oplus l}$ where $l$ is some finite positive integer, we get
\[
  \xymatrix{
    R\ar[d]_{\phi^{\vee}} & \ahom_R(S,R)\ar[l]\ar[d] \\
    \dsm{n_1} & \dsm{n_0}\ar[l]^{\nu^{\vee}}
  }.
\]
% where $\nu^{\vee},\phi^{\vee}$ are the transpose of $\nu,\phi$ respectively.
If we write $(\nu^{\vee})\dsm{n_1}$ for the submodule generated by the set of column vectors of $\nu^{\vee}$, then we have that the extension $R\rightarrow S$ is $\epf$-phantom if $\phi^{\vee}\in \left(\left( \nu^{\vee} \right)\dsm{n_1}\right)^{\epf}$ \cite[Lemma 2.10]{Dietz2010}. Note that our diagrams here do not completely match those in \cite[Discussion 2.4]{Dietz2010}. However, since whether an element in $\ext^1_R(S/R,R)$ is phantom or not is independent of the choice of the resolution of $S/R$ \cite[Discussion 2.3]{Dietz2010}, the test given here is also valid.

Next we discuss some perfectoid constructions we need.
\begin{cons}\label{m-cons-BhattPerfectoidConstruction}
  Let $K$ be a perfectoid field and let $(T[\frac{1}{t}],T)$ be a perfectoid affinoid $K$-algebra, where $t\in \kc$ is some uniformizer that lifts to $K^{\circ\flat}$, i.e., it admits a compatible system of $p$-power roots in $\kc$. Suppose that $g\in T$ lifts to $T^{\flat}$. Let $X$ denote the adic spectrum attached to the pair $(T[\frac{1}{t}],T)$, i.e., $X=\spa(T[\frac{1}{t}],T)$. Let $U_n$ be the rational subset $X \left( \dfrac{t^n}{g} \right)$. Then $\OO_X^+(U_n)$ is $\pir{t}$-almost isomorphic to the $t$-adic completion of $T[\left( \dfrac{t^n}{g}\right)^{1/p^{\infty}}]$. See also \cite[Lemma 6.4]{Scholze2012}.
\end{cons}

We also need a corollary of the quantitative form of Scholze's Hebbarkeitssatz (the Riemann extension theorem) for perfectoid spaces. First we state the theorem. See \cite[Theorem 4.2]{Bhatt2017}, or an alternative description in \cite[Theorem 11.2.1]{Bhatt2017notes}.
\begin{thm}\label{m-thm-QuantitativeRiemannHebbarkeitssatz}
  Let $(T[\frac{1}{t}],T),g\in T,X,U_n$ be as above. For each $m\geqslant 0$, assume that $g\in T$ is a nonzerodivisor in $T/t^mT$. Then the projective system of natural maps
  \[
    \set{f_n:T/g^m\to \OO_X^+(U_n)/g^m }
  \]
  is an almost-pro-isomorphism. In fact, we have the following more precise pair of assertions:
  \begin{enumerate}
  \item The kernels $\aker(f_n)$ are almost $0$ for each $n\geqslant 0$.
  \item For each $k\geqslant 0$ and $c\geqslant p^km$, the transition map $\coker(f_{n+c})\to \coker(f_n)$ has image almost annihilated by $g^{1/p^k}$.
  \end{enumerate}
\end{thm}

What we really need is the following corollary \cite[Corollary 11.2.2]{Bhatt2017notes}:
\begin{cor}\label{m-cor-QuantitativeRiemannHebbarkeitssatz}
  Let $(T[\frac{1}{t}],T),g\in T,X,U_n$ be as above. Assume that $g\in T$ is a nonzerodivisor in $T/t^mT$ for any $m$. Then for any $T$-complex $Q$, any integer $m\geqslant 0$ and any integer $i$, the natural map
  \[
    \ext^i_T(Q,T/t^mT)\to \varprojlim_n \ext_T^i(Q,\OO_X^+(U_n) / t^m \OO_X^+(U_n))
  \]
  has kernel and cokernel annihilated by $\pir{(tg)}$.
\end{cor}

The proof uses \cref{m-thm-QuantitativeRiemannHebbarkeitssatz} above and \cref{m-lem-AlmostProIsomoprhism} from \cref{sn-Preliminaries}.

We are ready to prove the main theorem of this section.
\begin{proof}[Proof of \cref{m-thm-EPFPhantom}]
  Let $\alpha\in \ext^1_R(S/R,R)$ be the obstruction to split $R\to S$. For any $R$-algebra $T$, we shall write $\alpha_T\in\ext_T^1((S/R)\otimes_R T,T)$ for the corresponding obstruction to splitting the induced map $T\to T\otimes_R S$.

  We know that $R$ is a complete local domain of mixed characteristic $p$ with an $F$-finite residue field. By Cohen's structure theorem, we have a module-finite extension $A\to R$ for some unramified complete regular local ring $A$. Since all maps in $A\to R\to S$ are module-finite extensions, and fraction fields have characteristic $0$, there is some $g\in A$ such that $A_{pg}\rightarrow R_{pg}\rightarrow S_{pg}$ are all finite \'{e}tale extensions.

  Let $A_{\infty,0}$ be constructed as in \cref{ssn-PerfectoidAlgebra} and let $A_{\infty}$ be an integral perfectoid algebra containing a system of compatible $p$-power roots of $g$ by \cref{m-thm-pPowerRoots}. Since $\cp{R}$ is a integral perfectoid $\kc$-algebra admitting a compatible system of $p$-power roots of $g$, there is a map $A_{\infty}\rightarrow \cp{R}$ of perfectoid $\kc$-algebras.

  Since $A_{\infty}$ is a perfectoid $K$-algebra, apply \cref{m-cons-BhattPerfectoidConstruction} to $X_A:=\spa(A_{\infty}[\frac{1}{p}],A_{\infty})$ and write $\cA_n:=\OO_{X_A}^+(X_A(\frac{p^n}{g}))$. Similarly apply the construction to $Y:=\spa(\cp{R}[\frac{1}{p}],\cp{R})$ and write $\cR_n:=\OO_Y^+(Y(\frac{p^n}{g}))$. Since $A_{\infty}\to \cp{R}$, we have an induced map $Y\to X_A$, which in turn induces maps $\cA_n\to\cR_n$ of perfectoid $\kc$-algebras.

  Fix an integer $n$. Since $A_{pg}\rightarrow R_{pg}$ is a finite \'etale extension, the base change map $\cA_n[\frac{1}{pg}]\rightarrow \cA_n\otimes_{A_{pg}}R_{pg}$ is also a finite \'etale extension. Note that $g$ is already inverted in $\cA_n[\frac{1}{p}]$. So $\cA_n[\frac{1}{p}]\rightarrow \cA_n\otimes_{A_p}R[\frac{1}{p}]$ is finite \'etale. By the almost purity theorem \cite[Theorem 7.9 (iii)]{Scholze2012}, the integral closure of $\cA_n\otimes_A R$ in $\cA_n\otimes_{A_p}R[\frac{1}{p}]$,  denoted $\cB_n$, is almost finite \'etale over $\cA_n$. $\cB_n$ is an $R$-algebra. Since both $\cA_n$ and $R[\frac{1}{p}]$ map to $\cR_n[\frac{1}{p}]$, their tensor product over $A$ also maps to $\cR_n[\frac{1}{p}]$. This induces a map between power-bounded elements, that is, $\cB_n\rightarrow \cR_n$. By the same argument applied to $R_{pg}\rightarrow S_{pg}$, the map $\cB_n[\frac{1}{p}]\rightarrow \cB_n\otimes_R S[\frac{1}{p}]$ is a finite \'etale extension. Thus, by the almost purity theorem, there is an $S$-algebra $\cC_n$ almost finite \'etale over $\cB_n$ and it almost splits. Therefore, the map $\cB_n\rightarrow \cB_n\otimes_R S\rightarrow \cC_n$ almost splits.

  Hence $\cB_n\to \cB_n\otimes_R S$ almost splits modulo $p^m$ for any $m$. In particular, $\alpha_{\cB_n/p^m\cB_n}$ is almost zero for all $n$ and $m$. Since we have a $\pir{p}$-almost map $\cB_n\to\cR_n$, we also have that $\alpha_{\cR_n/p^m\cR_n}$ is $\pir{(p^2g)}$-almost zero, hence, $\pir{(pg)}$-almost zero. By \cref{m-cor-QuantitativeRiemannHebbarkeitssatz}, we learn that $\alpha_{\cp{R}/p^m\cp{R}}$ is annihilated by $\pir{(pg)}$ for all $m$.

  Consider the construction discussed above:
  \[
    \xymatrix{
      0\ar[r] & R\ar[r] & S\ar[r] & S/R\ar[r] & 0 \\
      & \dsm{n_1}\ar[r]^{\nu}\ar[u]^{\phi} & \dsm{n_0}\ar[r]\ar[u] & S/R\ar[r]\ar[u] & 0
    }
  \]
  where $\phi$ is a $1\times n_1$ matrix with entries in $\mm$ and $\nu$ is a $n_0\times n_1$ matrix with entries in $\mm$.
  Then the image of $\phi^{\vee}$ in $\dsm[(\cp{R}/p^m\cp{R})]{n_1}$ is $\pir{(pg)}$-almost in the image of $\nu^{\vee}\dsm[(\cp{R}/p^m\cp{R})]{n_1}$. Noting that $\cp{R}/p^m\cp{R}\cong R^+/p^mR^+$, we have
  \[
    \pir{(pg)}\phi^{\vee} \in (\nu^{\vee})\dsm[(R^+)]{n_1} + p^m\dsm[(R^+)]{n_1}
  \]
  for any $m$. By the definition of $\epf$ closure (\cref{i-def-EpfModule}), we have $\phi^{\vee} \in \left((\nu^{\vee})\dsm[R]{n_1}\right)^{\epf}$. Hence, $R\rightarrow S$ is $\epf$-phantom by \cite[Lemma 2.10]{Dietz2010}.
\end{proof} 

\section{The Positive Characteristic Case}\label{sn-p-PositiveCharCase}
In this section, we discuss an analogue of \cref{m-thm-pColonCap} in positive characteristic. Since $p=0$, the situation is slightly different. We will use an arbitrary element instead, and the proof techniques are quite different. The main result is \cref{p-thm-ColonCap}.

\subsection{Intersection of ideals and regular sequences}
We begin by investigating the behaviour of intersections of finitely generated ideals in some non-noetherian rings.
\begin{prop}\label{p-prop-ParamIdealIntersect}
  Let $(R,\mm)$ be an excellent local domain of prime characteristic $p$. Let $R^+$ be its absolute integral closure. Suppose that $\seq{d}$ is a system of parameters of $R$. Then for any proper finitely generated ideal $J$ in $R^+$ and $\seq{n}$, part of a system of parameters, we have
  \[
    \bigcap_{N=1}^{\infty}((x_1,\ldots,x_n)+J^N)R^+ = (x_1,\ldots,x_n)R^+.
  \]
\end{prop}
\begin{proof}
  Suppose $J$ is generated by $\seq[r]{h}$. For any $u\in \bigcap_{N=1}^{\infty} ((x_1,\ldots,x_n)+J^N)R^+ $, let $S\subseteq R^+$ be a module-finite extension domain of $R$ containing $\seq{n},\seq[r]{h}$ and $u$. Then $S$ is also excellent. Let $J_0=(\seq[r]{h})S$. Then $u\in \left((\seq{n})S+J_0^NS\right)^+$ for any $N$. So $u\in \left((\seq{n})S+J_0^NS\right)^*$. By a well-known result \cite[Theorem 6.1]{Hochster1994b}, we can find a test element $c$ such that for any $q=p^e$, where $e$ is any positive integer,
  \[
    cu^q\in (x_1^q,\ldots,x_n^q)S+(J_0^N)^{[q]}S.
  \]
  Fix $q$. Since we know that $S/(x_1^q,\ldots,x_n^q)$ is $J_0$-adically separated, we have 
  \[
  \bigcap_{N=1}^{\infty}\left((x_1^q,\ldots,x_n^q)S+(J_0^N\right)^{[q]}S)\subseteq (x_1^q,\ldots,x_n^q)S.
  \]
  So we have
  \[
    cu^q\in (x_1^q,\ldots,x_n^q)S,
  \]
  which shows that $u\in \left((\seq{n})S\right)^{*}$. But $\left((\seq{n})S\right)^{*}=\left((\seq{n})S\right)^+$ by \cite[Theorem 5.1]{Smith1994}. So $u\in (x_1,\ldots,x_n)S^+=(x_1,\ldots,x_n)R^+$.
\end{proof}

In order to have enough different regular sequences, we need to adjoin new variables to $R$ and $R^+$. Hence, we introduce the following notation. Let $\Lambda$ be a possibly infinite index set and let $\set{t_{\lambda}:\lambda\in\Lambda}$ be a set of new variables. For any quasilocal ring $(T,\mm_T)$, the notion $T(t_{\lambda}:\lambda\in\Lambda)$ means the localization of the polynomial ring $T[t_{\lambda}:\lambda\in\Lambda]$ at the ideal $\mm_T T[t_{\lambda}:\lambda\in\Lambda]$. The natural map $T\rightarrow T(t_{\lambda}:\lambda\in\Lambda)$ is faithfully flat. Next, we prove a stronger version of \cref{p-prop-ParamIdealIntersect}.

\begin{prop}\label{p-prop-FFE-ParamIdealIntersect}
  Suppose $(R,\mm)$ is a complete local domain of prime characteristic $p$. Let $(R^+,\mm_{R^+})$ be its absolute integral closure. Suppose that $\seq{d}$ is a system of parameters of $R$. Then for any index set $\Lambda$ and the set of new variables $\set{t_{\lambda}:\lambda\in\Lambda}$, any proper finitely generated ideal $J$ in $R^+$ and $\seq{n}$, part of system of parameters, we have
  \[
    \bigcap_{N=1}^{\infty}((\seq{n})+J^N)T = (x_1,\ldots,x_n)T,
  \]
  where $T=\rpptl$.
\end{prop}
\begin{proof}
  Write $\mf{a}_N$ for the ideal $((\seq{n})+J^N)R^+$. Let $u$ be an element in $\bigcap_{N=1}^{\infty}\mf{a}_NT$. Then $u$ is a rational function in $t_{\lambda}$'s. Clear the denominator and assume that $u$ is a polynomial in $\rpbtl$. For each $N$, we know that $u\in\mf{a}_NT$, and, again by clearing denominators, there is some polynomial $g_N\in \rpbtl$ such that $g_Nu\in \mf{a}_N\rpbtl$. We also know that $g_N$ is a unit in $T$. Hence $g_N\not\in \mm_{R^+} \rpbtl$, which means that at least one coefficient of $g_N$ is not in $\mm_{R^+}R^+$, i.e., is a unit in $R^+$. Since $\mf{a}_N\subseteq\mm_{R^+}$, that coefficient continues to be a unit in $\rpbtl/\mf{a}_N\rpbtl$. Hence $g_N$ is actually a nonzerodivisor on $\rpbtl/\mf{a}_N\rpbtl$. Therefore, we conclude that $u\in \mf{a}_N\rpbtl$. Since we are now in the polynomial case, this is equivalent to saying that each coefficient of $u$ is in $\mf{a}_NR^+$. Therefore, each coefficient of $u$ is in $\cap_{N=1}^{\infty}\mf{a}_NR^+\subseteq (\seq{n})R^+$ (\cref{p-prop-ParamIdealIntersect}). So $u\in (\seq{n})\rpbtl\Rightarrow u\in (\seq{n})T$.
\end{proof}

For the proof of the main theorem of this section, we need to study the intersection of ideals containing elements of the form $x_n-t_{\lambda_n}x_{n+1}$ where $x_i$'s are part of a system of parameters.

\begin{prop}\label{p-prop-RS-PolynomialIntersect}
  Let $(B,\mm_B)$ be a quasilocal ring, $x, y\in \mm_B$ a permutable regular sequence on  $B$ and $t$ an indeterminate. Then for any positive integer $N$, we have
  \begin{enumerate}
  \item $(x-yt,y^N)B(t)\cap B[t]=(x-yt,y^N)B[t]$;
  \item $(x-yt,y^N)B[t]\cap B=(x,y)^NB$.
  \end{enumerate}
\end{prop}
\begin{proof}
  For (1), note that $(x-yt,y)B(t)=(x,y)B(t)$. Any polynomial $g(t)\in B[t]$ that is invertible in $B(t)$ has a nonzero coefficient that is a unit. Hence $g(t)$ is a nonzerodivisor on $B[t]/(x,y)B[t]$. So $y,x-yt,g(t)$ form a regular sequence on $B[t]$, which implies that $y^N,x-yt,g(t)$ form a regular sequence on $B[t]$. Suppose that $p(t)\in (x-yt,y^N)B(t)\cap B[t]$; then we can clear the denominators to get $ p(t)g(t)=a(t)y^N+b(t)(x-yt) $ for some $a(t),b(t),g(t)\in B[t]$ with $g(t)$ invertible in $B(t)$. Then since $y^N,x-yt,g(t)$ form a regular sequence on $B[t]$, we conclude that $p(t)\in (x-yt,y^N)B[t]$. The converse containment is obvious.

  For (2), suppose that $b\in (x-yt,y^N)B[t]\cap B$. Then $b=\alpha(t) (x-yt)+\beta(t) y^N$ for some polynomial $\alpha(t),\beta(t)\in B[t]$. Assume that $\alpha(t)=c_ht^h+\cdots+c_0$ for elements $c_0,\seq[c]{h}\in B$. By comparing the coefficients of $b=\alpha(t) (x-yt)+\beta(t) y^N$, we have
  \begin{align}
    -c_hy &\in (y^N)B, \label{eq:p-tpcc-1}\\
    -c_{i-1}y+c_ix &\in (y^N)B, \quad 1\leqslant i\leqslant h, \label{eq:p-tpcc-2}\\
    b-c_0x&\in (y^N)B.\label{eq:p-tpcc-2.5}
  \end{align}

  \begin{clm*}
    For each coefficient, we have $c_{h-i}\in (x,y)^{N-1}B$ for $0\leqslant i\leqslant h$.
  \end{clm*}
  We prove this by induction on $i$. As we have $c_h\in (y^{N-1})B$ from \eqref{eq:p-tpcc-1} as $y$ is, by assumption, a nonzerodivisor on $B$, the case $i=0$ is obvious.

  Assume that the claim is true for $c_{h-i+1}$, i.e., $c_{h-i+1}\in (x,y)^{N-1}B$. So we write 
  \[
  c_{h-i+1} = \gamma_{N-1}x^{N-1}+\gamma_{N-2}x^{N-2}y+\cdots+\gamma_0y^{N-1}
  \]
  for some $\gamma_0,\seq[\gamma]{N-1}\in B$.

  By \eqref{eq:p-tpcc-2} we have
  \[
    -c_{h-i}y+c_{h-i+1}x =y^N\mu,
  \]
  for some $\mu\in B$. Hence we have
  \begin{align*}
    c_{h-i}y = \gamma_{N-1}x^N+\gamma_{N-2}x^{N-1}y+\cdots+\gamma_0xy^{N-1}&-\mu y^N \\
    \Rightarrow  \left(c_{h-i}+ \mu y^{N-1}-\gamma_0 xy^{N-2}-\gamma_1x^2y^{N-3}-\cdots - \gamma_{N-2}x^{N-1} \right) y &\in (x^N)B.
  \end{align*}

  Using the assumption that $y$ is a nonzerodivisor on $B/(x^N)B$, we conclude that
  \[
    c_{h-i} \in (x^N,x^{N-1}, x^{N-2}y,\cdots,xy^{N-2},y^{N-1})B = (y^{N-1},xy^{N-2},...,x^{N-1})B = (x,y)^{N-1}B.
  \]

  Therefore, the claim is proved. 
  
  By the claim above, we have $c_0\in (x,y)^{N-1}B\Rightarrow c_0x\in (x,y)^NB$. So \eqref{eq:p-tpcc-2.5} implies that $b\in (x,y)^NB$.
%   Each coefficient of $\alpha(t)$ is in $ (x,y)^{N-1}B$. Then each coefficient of $\alpha(t)(x-yt)$ is in $(x,y)^NB$. Hence $\alpha(t)(x-yt)\in (x,y)^NB[t]$. Obviously $\beta(t) y^N\in(x,y)^NB[t]\Rightarrow b\in (x,y)^NB[t]\cap B$. Since $B[t]$ is flat over $B$, we have $b\in (x,y)^NB$.

  Conversely, it is trivial that $y^N\in (x-yt,y^N)B[t]$. If $x^ky^{N-k}\in (x-yt,y^N)B[t]$ for some $0\leqslant k\leqslant N$, then $x^{k+1}y^{N-k-1}=(x-yt)x^ky^{N-k-1}+x^ky^{N-k}t\in (x-yt,y^N)B[t]$. So inductively we have $(x,y)^N\subseteq (x-yt,y^N)B[t]$. Hence they are equal.
\end{proof}

%It is straightforward to prove following lemma.
We next observe that
\begin{lem}\label{p-lem-RS-FaithfullyFlat}
  Let $(R,\mm,K)$ be a $d$-dimensional complete local domain of prime characteristic $p$. Let $R^+$ be its absolute integral closure. Let $\Lambda$ be an index set and let $\set{t_{\lambda}:\lambda\in\Lambda}$ be a set of variables. Let $\seq{d}$ be a system of parameters of $R$. Then  $\seq{n}$, where $1\leqslant n\leqslant d$, is a regular sequence on $T=\rpptl$:
\end{lem}
\begin{proof}
  We know that $R^+$ is a big Cohen-Macaulay $R$-algebra, so $\seq{d}$ is a regular sequence on $R^+$. Since $R^+\rightarrow T$ is faithfully flat, we know that $x_1$ is a nonzerodivisor in $T$ and
  \[
    (\seq{n-1})T:_T x_n = ((\seq{n-1})R^+:_{R^+} x_n)T = (\seq{n-1})T.
  \]
  Therefore, $\seq{d}$ is a regular sequence on $T$.
\end{proof}

The following results are standard facts about regular sequences.
\begin{lem}\label{p-lem-RS-Transposition}
  Let $B$ be a ring and $x,y\in B$ be a regular sequence on $B$. If $y$ is a nonzerodivisor on $B$, then $y,x$ is also a regular sequence.
\end{lem}
\begin{proof}
  The only thing that needs a proof is that $x$ is a nonzerodivisor on $B/(y)B$. Suppose that $\alpha x=\beta y$ for some $\alpha,\beta\in B$. Then by assumption $\beta\in (x)B$. So we can write $\beta = x\beta'$ for some $\beta'\in B$. Then $x(\alpha-\beta' y)=0$ and $x$ is a nonzerodivisor. So $\alpha = \beta' y\Rightarrow \alpha \in (y)B$.
\end{proof}

\begin{cor}\label{p-cor-RS-Permutability}
  Let $B$ be a ring and $y,\seq{n}$ be elements of $B$. If both $\seq{n}$ and $y,\seq{n}$ are regular sequences on $B$, then so is $\seq{k},y,x_{k+1},\ldots,x_n$ where $0\leqslant k\leqslant n$. In particular, when $k=n$, $\seq{n},y$ form a regular sequence on $B$.
\end{cor}
\begin{proof}
  We prove this by induction on $k$. The base case $k=0$ is one of the assumptions. Now assume that $\seq{k},y,x_{k+1},\ldots,x_n$ is a regular sequence on $B$. Let $A=B/(\seq{k})$. We know that $y,x_{k+1}$ form a regular sequence, and $x_{k+1}$ is a nonzerodivisor on $A$. By \cref{p-lem-RS-Transposition}, $x_{k+1},y$ also form a regular sequence on $A$. It is obvious that $x_{k+2},\ldots,x_n$ continue to be a regular sequence on $A/(x_{k+1},y)$. Therefore $\seq{k+1},y,x_{k+2},\ldots,x_n$ is a regular sequence on $B$.
\end{proof}

The main result we want to prove is about elements of the form $x_n-t_{\lambda_n}x_{n+1}$ in the ring $T$. %In the theorem below, if the index set for a variable is empty, then the variable does not occur.
\begin{thm}\label{p-thm-RSonT}
  Let $(R,\mm,K)$ be a $d$-dimensional complete local domain of prime characteristic $p$. Let $R^+$ be its absolute integral closure. Let $\Lambda$ be an index set and let $\set{t_{\lambda}:\lambda\in\Lambda}$ be a set of variables. Let $T=\rpptl$. For any system of parameters $\seq{d}$ of $R$, let $z_i=x_i-t_{\lambda_i}x_{i+1}$ for some $t_{\lambda_i}\in \set{t_{\lambda}:\lambda\in \Lambda}$ ($\lambda_i\neq\lambda_j$ if $i\neq j$ and $1\leqslant i<d$) and $\set{\seq[y]{h}},\set{\seq[w]{l}}$ be two subsets of $X_n:=\set{x_{n+2},x_{n+3},...,x_d}$ where $0\leqslant h,l\leqslant d-n-1$. %($h=0$ (resp. $l=0$) means $\set{\seq[y]{h}}=\emptyset$ (resp. $\set{\seq[w]{l}}=\emptyset$)).
  Then we have
  \begin{enumerate}
  \item $\seq[y]{h},\seq[z]{n},x_{n+1}$ form a regular sequence on $T$, and
  \item $\bigcap_{N=1}^{\infty}\left( (\seq[y]{h},\seq[z]{n})+(x_{n+1},\seq[w]{l})^N\right)T=\left( \seq[y]{h},\seq[z]{n} \right)T$
  \end{enumerate}
  for any $0\leqslant n\leqslant d-1$. % ($n=0$ means that there is no $z_i$ involved).
\end{thm}
\begin{proof}
  We prove both claims at the same time by induction on $n$. The base case $n=0$ is trivial: when $n=0$, (1) is the conclusion of \cref{p-lem-RS-FaithfullyFlat} and (2) is the conclusion of \cref{p-prop-FFE-ParamIdealIntersect}.

  Assume that both claims are true for $n$. We want to prove the case of $n+1$. So we let both $\set{\seq[y]{h}}$ and $\set{\seq[w]{l}}$ be subsets of $X_{n+1}$. We first prove (1). Note that $X_{n+1}\subseteq X_n$. So $\seq[y]{h},\seq[z]{n},x_{n+1}$ is a regular sequence by the induction hypothesis. Since $(\seq[z]{n},x_{n+1})T=(\seq[x]{n+1})T$, it is trivial that $\seq[y]{h},\seq[z]{n},x_{n+1},x_{n+2}$ form a regular sequence on $T$. Let $S=T/(\seq[y]{h},\seq[z]{n})T$. If there is some $\alpha\in S$ such that $\alpha z_{n+1}=0$, then modulo $x_{n+1}$ we see that $\alpha t_{n+1}x_{n+2}=0$. Hence $\alpha = x_{n+1}\beta$. Since $x_{n+1}$ is a nonzerodivisor and $\beta x_{n+1} z_{n+1}=0$, we have $\beta z_{n+1}=0$. Repeating this argument we get $\alpha\in (\seq[y]{h},\seq[z]{n},x_{n+1}^N)T$ for all $N$. The induction hypothesis (2) shows that $\alpha \in (\seq[y]{h},\seq[z]{n})T$. Hence $z_{n+1}$ is a nonzerodivisor on $S$.

  Let $T'=T/(\seq[y]{h})T$. Since $x_{n+2}\in X_n$, the induction hypothesis shows that both $x_{n+2},\seq[z]{n}$ and  $\seq[z]{n} $ are regular sequences on $T'$. By \cref{p-cor-RS-Permutability}, $\seq[z]{n},x_{n+2}$ also form a regular sequence on $T'$. From the previous paragraph, we know that $\seq[z]{n},x_{n+1},x_{n+2}$ is a regular sequence on $T'$, and, again by \cref{p-cor-RS-Permutability}, $\seq[z]{n},x_{n+2},x_{n+1} $ form a regular sequence on $T'$. So $\seq[z]{n},x_{n+2},x_{n+1}-t_{n+1}x_{n+2}$ is a regular sequence on $T'$. Again from the previous paragraph, $\seq[z]{n},z_{n+1}$ is already a regular sequence on $T'$. We conclude that $ \seq[y]{h},\seq[z]{n},z_{n+1},x_{n+2} $ form a regular sequence on $T$ by \cref{p-cor-RS-Permutability}. This proves (1).

  For (2), let $Q=R^+(t_{\mu}:\mu\in\Lambda,\mu\neq\lambda_{n+1})$. Then $T=Q(t_{\lambda_{n+1}})$. % The map $Q\rightarrow T$ is intersection-flat(?) and faithfully flat.
  Since  $(x_{n+2}^N,\seqp{w}{l})Q$ is cofinal with $(x_{n+2},\seq[w]{l})^NQ$, it suffices to show that
  \[
    \bigcap_{N=1}^{\infty}((\seq[y]{n},\seq[z]{n+1})+(x_{n+2}^N,w_1^N,...,w_l^N))Q(t_{\lambda_{n+1}}) \subseteq (\seq[y]{h},\seq[z]{n+1})Q(t_{\lambda_{n+1}}).
  \]
  Since $\set{\seq[w]{l}}\cup \set{\seq[y]{h}} \subseteq X_{n+1}$, the ideal $(\seq[y]{h},\seqp{w}{l})\subseteq (x_{n+3},\ldots,x_d)$ is generated by part of a system of parameters of $R$. Hence $x_{n+1},x_{n+2}$ form a permutable regular sequence on the quotient ring $P = Q/(\seq[y]{h},\seqp{w}{l},\seq[z]{n})Q$. Applying \cref{p-prop-RS-PolynomialIntersect} to $B=P,x=x_{n+1},y=x_{n+2}$, we get
  \begin{align*}
    (\seq[y]{h},\seqp{w}{l},&\seq[z]{n},x_{n+1}-t_{\lambda_{n+1}}x_{n+2},x_{n+2}^N)T \cap Q \\
                            & =(\seq[y]{h},\seqp{w}{l},\seq[z]{n})Q+(x_{n+1},x_{n+2})^NQ.
  \end{align*}

  For any $u\in \cap_{N=1}^{\infty}(\seq[y]{h},\seqp{w}{l},\seq[z]{n},x_{n+1}-t_{\lambda_{n+1}}x_{n+2},x_{n+2}^N)T$, clear the denominators and assume that $u$ is a polynomial in $t_{\lambda_{n+1}}$ of degree $h$. Then $x_{n+2}^hu$ can be considered as a polynomial in $t_{\lambda_{n+1}}x_{n+2}$ of degree $h$. Therefore, we can divide $x_{n+2}^hu$ by the ``monic'' polynomial $t_{\lambda_{n+1}}x_{n+2}-x_{n+1}$ in $Q[t_{\lambda_{n+1}}]$ and get a remainder $b$ of degree $0$ in $t_{\lambda_{n+1}}$. So $x_{n+2}^hu-b\in (z_{n+1})Q[t_{\lambda_{n+1}}]$ and
  \begin{align*}
  b&\in \bigcap_{N=1}^{\infty}(\seq[y]{h},\seqp{w}{l},\seq[z]{n+1},x_{n+2}^N)T \cap Q \\
  \Rightarrow b&\in \bigcap_{N=1}^{\infty}((\seq[y]{h},\seqp{w}{l},\seq[z]{n})+(x_{n+1},x_{n+2})^N)Q,
  \end{align*}
  which implies that $b\in (\seq[y]{h},\seq[z]{n})Q$ by the induction hypothesis. Therefore, 
  \[
  x_{n+2}^hu\in (\seq[y]{h},\seq[z]{n+1})Q[t_{\lambda_{n+1}}]\Rightarrow x_{n+2}^hu\in (\seq[y]{h},\seq[z]{n+1})T.
  \]
  Note that $\seq[y]{h},\seq[z]{n+1},x_{n+2}^h$ is a regular sequence on $T$ by our proof of (1) above. Hence, we do not need the factor of $x_{n+2}^h$ and we have $u\in (\seq[y]{h},\seq[z]{n+1})T$.
\end{proof}

%%%%%%%%%%%%%% 
\subsection{Stabilization of colon ideals}
We say that the colon ideal $(\seq{n})R^+:_{R^+}y^{\infty}$ \emph{stabilizes} at a positive integer $N$ if
\[
  (\seq{n})R^+:_{R^+}y^{\infty}=(\seq{n})R^+:_{R^+}y^N.
\]
The next few results deal with the stability of colon ideals in non-noetherian rings.
\begin{lem}\label{p-lem-FullParaColonIdeal}
  Let $(R,\mm)$ be a $d$-dimensional local domain of prime characteristic $p$ and $T$ an $R$-algebra. Let $\seq{d}$ be a system of parameters of $R$. Suppose that $x_d$ is a nonzerodivisor on $T/(\seq{d-1})T$. Then for any element $y\in R$ and $n=d,d-1$, there is some $N_0$ such that
  \[
    (\seq{n})T:_Ty^{\infty} = (\seq{n})T:_Ty^{N_0}.
  \]
\end{lem}
\begin{proof}
  The conclusion is trivial if $y$ is a unit. So we assume that $y\in\mm$. For $n=d$, since $\mm$ is nilpotent on $(\seq{d})R$, and $y\in\mm$, there is some $N_0$ such that $y^{N_0}\in (\seq{d})R$. This $N_0$ will suffice.

  Look at $n=d-1$ and consider the ring $A=(R/(\seq{d-1})R)_{x_d}$. It is an artinian ring, and hence it is a product of artinian local rings, i.e., $A=A_1\times \cdots\times A_h$. The image of $y$ in each component $A_i$ is either a unit or a nilpotent, i.e., $\bar{y}=(\seq[y]{h})$ where, without loss of generality, $\seq[y]{k}$ are nilpotents and $y_{k+1},\ldots,y_h$ are units. So there is some positive power $N_0$ such that $\bar{y}^{N_0}=(0,0,\ldots,0,y_{k+1}^{N_0},\ldots,y_h^{N_0})$. There is some $a\in A$ such that
  \begin{equation} \label{eq:p-tpcc-3}
    \bar{y}^{N_0}=a\bar{y}^{2N_0}.
  \end{equation}
  Since $T$ is an $R$-algebra, we have a map $A\rightarrow B=(T/(\seq{d-1})T)_{x_d}$. The relation \eqref{eq:p-tpcc-3} of $\bar{y}^{N_0}$ and $\bar{y}^{2N_0}$ maps to a relation
  \begin{equation} \label{eq:p-tpcc-4}
    \bar{y}^{N_0}=b\bar{y}^{2N_0}
  \end{equation} 
  for some $b\in B$ and the same $N_0$. We claim that $(\seq{d-1})T_{x_d}:_{T_{x_d}}y^{\infty}=(\seq{d-1})T_{x_d}:_{T_{x_d}}y^{N_0}$. Take any $u\in (\seq{d-1})T_{x_d}:_{T_{x_d}}y^{\infty}$. Then the image $\bar{u}$ of $u$ in $B=(T/(\seq{d-1})T)_{x_d}$ is in $0:_B\bar{y}^{\infty}$. So there is some $N$ such that $\bar{y}^N\bar{u}=0$ in $B$. It is obvious that any higher power of $\bar{y}$ will kill $\bar{u}$. Hence we may assume without loss of generality that $N= m N_0$ and $m\geqslant 2$. Then $ \bar{y}^{m N_0} \bar{u} = 0$.
  Making use of the relation \eqref{eq:p-tpcc-4}, we have
  \begin{align*}
    b\bar{y}^{2N_0} \bar{y}^{(m -2)N_0} \bar{u}& =0 \\
    \Rightarrow   \bar{y}^{N_0} \bar{y}^{(m -2)N_0} \bar{u}&=0 \\
    \Rightarrow    \bar{y}^{(m -1)N_0} \bar{u}&=0.
  \end{align*}

  We can repeat this argument until $m$ reaches $1$. So $\bar{y}^{N_0}\bar{u}=0 $ in $B$ implies that $y^{N_0}u \in (\seq{d-1})T_{x_d}$, which in turn implies that $u\in (\seq{d-1})T_{x_d}:_{T_{x_d}}y^{N_0}$. So we have
  \[
    (\seq{d-1})T_{x_d}:_{T_{x_d}}y^{\infty}=(\seq{d-1})T_{x_d}:_{T_{x_d}}y^{N_0}.
  \]
  But we also know that $x_d$ is an nonzerodivisor on $T/(\seq{d-1})T$. So we have
  \[
    (\seq{d-1})T:_Ty^{\infty}=(\seq{d-1})T:_Ty^{N_0}.
  \]
  as well.
\end{proof}

\begin{thm}\label{p-thm-FFE-ParaColonPowerStable}
  Suppose $(R,\mm,K)$ is a $d$-dimensional complete local domain of prime characteristic $p$. Let $R^+$ be its absolute integral closure. Let $\Lambda$ be an uncountable index set and let $\set{t_{\lambda}:\lambda\in\Lambda}$ be a set of variables. Suppose $\seq[x]{d}$ is a system of parameters of $R$. Then for any $1\leqslant n\leqslant d$ and any $y\in R$, there is some $N_0$ such that
  \[
    (\seq[x]{n})T:_Ty^{\infty}=(\seq{n})T:_Ty^{N_0}
  \]
  where $T=\rptl$.
\end{thm}
\begin{proof}
  By \cref{p-lem-FullParaColonIdeal}, the conclusion is true for $n=d,d-1$. We assume that $n\leqslant d-2$. We also assume that $y\in\mm$. Let $z_i=x_i-t_{\lambda_i}x_{i+1}$ ($1\leqslant i\leqslant d-1$). Consider the sequence $\seq{n},z_{n+1},\ldots,z_{d-1},x_d$. It is easy to see that $(\seq{n},z_{n+1},\ldots,z_{d-1},x_d)R(t_{\lambda_{n+1}},\ldots,t_{\lambda_{d-1}})=(\seq{d})R(t_{\lambda_{n+1}},\ldots,t_{\lambda_{d-1}})$. So by \cref{p-lem-FullParaColonIdeal}, for any choice of $\lambda_{n+1},\ldots,\lambda_{d-1}$, there is some $N_0$ such that
  \[
    (\seq{n},z_{n+1},\ldots,z_{d-1})T:_Ty^{\infty}=(\seq{n},z_{n+1},\ldots,z_{d-1})T:_Ty^{N_0}.
  \]

  We want to show that
  \[
    (\seq{n},z_{n+1},\ldots,z_k)T:_Ty^{\infty}=(\seq{n},z_{n+1},\ldots,z_k)T:_Ty^{N_0}
  \]
  for $k\geqslant n$.
  We prove this by reverse induction on $k$. The base case $k=d-1$ is done. Now suppose that this is true for $k+1$.
  Look at ideals $(\seq{n},z_{n+1},\ldots,z_k,x_{k+1}-t_{\mu}x_{k+2})T:_T y^{\infty}$. The induction hypothesis shows that each ideal stabilizes at some $N$. There are uncountably many $\mu\in\Lambda$. So we can find some $N_0$ such that
  \begin{equation}\label{eq:2b2}
    (\seq{n},z_{n+1},\ldots, z_k,x_{k+1}-t_{\mu}x_{k+2})T:_T y^{\infty}=(\seq{n},z_{n+1},\ldots,z_k,x_{k+1}-t_{\mu} x_{k+2})T:_T y^{N_0}.
  \end{equation}
  holds for infinitely many choices of $\mu$. In particular, there are countably many $\mu_1,\mu_2,\ldots$ avoiding all $\lambda_{n+1},\ldots,\lambda_k$ such that \eqref{eq:2b2} holds.

  For any $u\in (\seq{n},z_{n+1},\ldots,z_k)T:_Ty^{\infty}$, there is some $N$ such that $y^Nu\in (\seq{n},z_{n+1},\ldots,z_k)T$. So
  \[
    y^Nu\in (\seq{n},z_{n+1},\ldots,z_k,x_{k+1}-t_{\mu_i} x_{k+2})T\Rightarrow y^{N_0}u\in (\seq{n},z_{n+1},\ldots,z_k,x_{k+1}-t_{\mu_i} x_{k+2})T
  \]
  for all choices of $\mu_i$.

  Hence for any $l$, we have 
  \[
  y^{N_0}u\in \cap_{i=1}^l (\seq{n},z_{n+1},\ldots,z_k,x_{k+1}-t_{\mu_i} x_{k+2})T.
  \]
  Let $a_i=x_{k+1}-t_{\mu_i}x_{k+2}$ and $S=T/(\seq{n},z_{n+1},\ldots,z_k)T$.
  We claim that
  \begin{equation}\label{eq:lbl}
    \bigcap_{i=1}^l (a_i)S =\prod_{i=1}^l (a_i)S  
  \end{equation}
  for any $l$.
  
  We first notice that any two elements $a_i$ and $a_j$ where $i\neq j\in\NN$ form a regular sequence in $S$: because $(a_i,a_j)S=(x_{n+1},x_{n+2})S$, they form a regular sequence on $S$. We prove \eqref{eq:lbl} by induction on $l$. The case $l=1$ is trivial. Suppose that $u\in \left( \cap_{i=1}^l (a_i)S\right) \cap (a_{l+1})S=\left(\prod_{i=1}^l (a_i)S\right)\cap (a_{l+1})S$. Write $c_l=\prod_{i=1}^la_i$. Then $u=\alpha c_l=\beta a_{l+1}$ for some $\alpha,\beta\in S$. Since $a_{l+1},a_i$ ($ i\neq l$) form a regular sequence, so do $a_{l+1},c_l$. Hence $\alpha\in (a_{l+1})S\Rightarrow u\in (a_{l+1}c_l)S$ and \eqref{eq:lbl} is proved.

  So for any $l$, we have
  \[
    y^{N_0}u \in (\seq{n},z_{n+1},\ldots,z_k)T+\prod_{i=1}^l (x_{k+1}-t_{\mu_i} x_{k+2})T.
  \]
  
  Since
  \begin{align*}
    \bigcap_{l=1}^{\infty}&\left((\seq{n},z_{n+1},\ldots,z_k)T+\prod_{i=1}^l (x_{k+1}-t_{\mu_i} x_{k+2})T\right) \\
                          & \subseteq \bigcap_{l=1}^{\infty}\left((\seq{n},z_{n+1},\ldots,z_k)T+\left((x_{k+1},x_{k+2})T\right)^l\right),
  \end{align*}
  we may apply \cref{p-prop-FFE-ParamIdealIntersect} to see that the right-hand side is in $(\seq{n},z_{n+1},\ldots,z_k)T$. We conclude that
  \[
    y^{N_0}\in (\seq{n},z_{n+1},\ldots,z_k)T.
  \]
  So we have
  \[
    (\seq{n},z_{n+1},\ldots,z_k)T:_T y^{\infty} = (\seq{n},z_{n+1},\ldots,z_k)T:_T y^{N_0}.
  \]
  The case $k=n$ is the conclusion of the theorem.
\end{proof}

\begin{rmk}
In \cref{p-thm-FFE-ParaColonPowerStable}, one can assume that $\Lambda$ is a countably infinite index set. The proof still works if we make the following modification: 
we observe that for two different variables $t_{\lambda}$ and $t_{\mu}$, the map swapping $t_{\lambda}$ and $t_{\mu}$ is an automorphism of $T=\rptl$. Hence, if we have
\[
    (\seq{n},z_{n+1},\ldots, z_k,x_{k+1}-t_{\lambda}x_{k+2})T:_T y^{\infty}=(\seq{n},z_{n+1},\ldots,z_k,x_{k+1}-t_{\lambda} x_{k+2})T:_T y^{N_0}
\]
for some $N_0$, then by applying the automorphism we just described, we have
\[
    (\seq{n},z_{n+1},\ldots, z_k,x_{k+1}-t_{\mu}x_{k+2})T:_T y^{\infty}=(\seq{n},z_{n+1},\ldots,z_k,x_{k+1}-t_{\mu} x_{k+2})T:_T y^{N_0}.
\]
So the proof where we show that there are countably many $\mu$ such that \eqref{eq:2b2} holds can be modified as follows: Look at ideals $(\seq{n},z_{n+1},\ldots,z_k,x_{k+1}-t_{\mu}x_{k+2})T:_T y^{\infty}$. The induction hypothesis shows that each ideal stabilizes at some $N$. If one such ideal stabilizes at some integer $N_0$, then by permuting the variables $t_{\mu}$ we see that all such ideals stabilize at the same $N_0$.
\end{rmk}

\begin{cor}\label{p-cor-ParaColonPowerStable}
  Suppose $(R,\mm,K)$ is a $d$-dimensional complete local domain of prime characteristic $p$. Let $R^+$ be its absolute integral closure. Suppose $\seq{d}$ is a system of parameters of $R$. Then for any $1\leqslant n\leqslant d$ and any $y\in R$, there is some $N_0$ such that
  \[
    (\seq{n})R^+:_{R^+}y^{\infty}=(\seq{n})R^+:_{R^+}y^{N_0}.
  \]
\end{cor}
\begin{proof}
  Let $\Lambda$ be an uncountable index set and $\set{t_{\lambda}:\lambda\in\Lambda}$ a set of new variables. Then by \cref{p-thm-FFE-ParaColonPowerStable}, there is some $N_0$ such that
  \[
    (\seq{n})\rpptl:_{\rpptl}y^{\infty}=(\seq{n})\rpptl:_{\rpptl}y^{N_0}.
  \]
  For any element $u \in (\seq{n})R^+:_{R^+}y^{\infty}$, there is some $N$ such that $y^Nu\in   (\seq{n})R^+$. So it is also in $ (\seq{n})\rpptl$. We have $y^{N_0}u\in (\seq{n})\rpptl$ by equality above. So we have
  \[
    y^{N_0}u\in (\seq{n})\rpptl\cap R^+ \Rightarrow     y^{N_0}u\in (\seq{n})R^+
  \]
  as the map $R^+\rightarrow \rpptl$ is faithfully flat. Hence $u\in (\seq{n})R^+:_{R^+}y^{N_0}$, as desired.
\end{proof}

\subsection{The main theorem} We are almost ready to prove our main theorem of this section. Before that, let us derive a useful corollary from the results on the stability of colon ideals.

\begin{cor}\label{p-cor-ParaIdealDecomp}
  Suppose $(R,\mm)$ is a $d$-dimensional complete local domain of prime characteristic $p$. Let $R^+$ be its absolute integral closure. Suppose $\seq{d}$ is a system of parameters of $R$. Then for any $1\leqslant n\leqslant d$ and any $y\in R$, there is some $N_0$ such that $(\seq{n})R^+ = \mf{a}\cap \mf{b}$ where $\mf{a}=(\seq{n},y^{N_0})R^+$ and $\mf{b}=(\seq{n}):_{R^+}y^{N_0}$.
\end{cor}
\begin{proof}
  By \cref{p-cor-ParaColonPowerStable}, there is some $N_0$ such that $(\seq{n})R^+:_{R^+}y^{\infty}=(\seq{n})R^+:_{R^+}y^{N_0}$. Let $\mf{a}=(\seq{n},y^{N_0})R^+$ and $\mf{b}=(\seq{n})R^+:_{R^+}y^{N_0}$. Then we have $\mf{b}=(\seq{n})R^+:_{R^+}y^{\infty}$. For any $u\in \mf{a}\cap \mf{b}$, we can write $u=a_1x_1+\cdots+a_nx_n+by^{N_0}$ and we know that $y^{N_0}u\in (\seq{n})R^+$. So we have $by^{2N_0}\in(\seq{n})R^+\Rightarrow b\in\mf{b}\Rightarrow by^{N_0}\in(\seq{n})R^+$. So $u\in (\seq{n})R^+$. The reverse inclusion is trivial.
\end{proof}

Now we prove the main theorem.
\begin{thm}\label{p-thm-ColonCap}
  Suppose $(R,\mm)$ is a $d$-dimensional complete local domain of prime
  characteristic $p$. Let $R^+$ be its absolute integral closure.
  Suppose $\seq{d}$ is a system of parameters of $R$. Then for any
  $1\leqslant n\leqslant d$ and any $y\in R$, there is some positive integer $N_0$
  such that for all $N\geqslant N_{0}$,
  \[
    (\seq{n},y^N)R^+:_{R^+}x_{n+1}\subseteq (\seq{n},y^{N-N_0})R^+.
  \]
\end{thm}
\begin{proof}
  By applying \cref{p-cor-ParaIdealDecomp} to the system of parameters
  $\seq{n+1}$ and $y$, we know that there is some $N_{0}$ such that if
  we write $\mf{a}=(x_1,\ldots.,x_{n+1},y^{N_0})R^+$ and $\mf{b}=(x_1,\ldots.,x_{n+1})R^+:_{R^+}y^{N_0}$,
  then we have
  $(x_1,\ldots,x_{n+1})R^+=\mf{a}\cap\mf{b}$. For any $u\in
  (x_1,\ldots,x_n,y^N)R^+:_{R^+}x_{n+1}$, we have
  $ux_{n+1}=u_1x_1+\cdots+u_nx_n+vy^N$ for some $\seq[u]{n},v\in R^+$.
  Therefore, $v\in (\seq{n+1})R^+:_{R^+}y^N\subseteq
  \mf{b}=(\seq{n+1})R^+:_{R^+}y^{N_0}$. So we can write $
  vy^{N_0}=v_1x_1+\cdots+v_{n+1}x_{n+1}$ for some $\seq[v]{n+1}\in R^+$.
  Hence,
  \begin{align*}
    (u-v_{n+1}y^{N-N_0})x_{n+1}&=(u_1+v_1y^{N-N_0})x_1+\cdots+(u_n+v_ny^{N-N_0})x_n \\
    \Rightarrow u-v_{n+1}y^{N-N_0} &\in (\seq{n})R^+:_{R^+}x_{n+1}.
  \end{align*}
  Since $R^+$ is a big Cohen-Macaulay algebra of $R$, we have $   u-v_{n+1}y^{N-N_0} \in  (\seq{n})R^+ $. Therefore, we have $ u \in (\seq{n},y^{N-N_0})R^+ $. 
\end{proof}
\begin{rmk}
  Since $R$ is complete local, by Cohen's structure theorem, $R$ is module-finite over some complete regular local domain $A$ of characteristic $p$ with respect to the system of parameters $\seq{d}$. If the element $y$ happens to be in $A$, then the same proof (i.e., the proof of \cref{m-thm-pColonCap}) as in the mixed characteristic case also works.
\end{rmk}
% end of section "char p case"

\section{Questions}
\subsection{More variations of \texorpdfstring{$\epf$}{epf} closure}
Instead of considering one closure operation, we can consider a family of closure operations as follow. Let $(R,\mm)$ be a complete local domain of mixed characteristic $p$ and $R^+$ its absolute integral closure. An element $u\in R$ is in the closure of $I\subseteq R$ if there is some nonzero element $c\in R$ such that for any $N\in\NN,\varepsilon\in\QQ^+$, we have
\[
  c^{\varepsilon}u\in IR^+ + J^NR^+,
\]
where $J$ is some fixed ideal in $R$ containing $p$ and contained in $\mm$. In particular, if we choose $J$ to be $pR$, then we recover the usual $\epf$ closure.

All these closure operations are potentially larger than the $\epf$ closure. So they all satisfy the usual colon-capturing property. Moreover, the proof that the $\epf$ closure is trivial on regular local rings \cite[Theorem 3.9]{Heitmann2018} works for this family of closures. So they are trivial on regular local rings. However, we do not know whether any of these closure operations (including the $\epf$ closure) satisfy the generalized colon-capturing axiom. So we ask
\begin{ques}
  Do any of the closure operations described above satisfy the generalized colon-capturing axiom?
\end{ques}

\subsection{Persistence}
It is also interesting to know if any of these closure operations has persistence. Recall that a closure operation $\cl$ satisfies persistence for change of rings if whenever $R\rightarrow S$ is a local morphism between complete local domains, and $N\subseteq M$ are finitely generated $R$-modules, then 
\[
\aim \left( S\otimes_R N_M^{\cl} \rightarrow S\otimes_R M \right) \subseteq \left( \aim \left( S\otimes_R N\rightarrow S\otimes_R M \right) \right)^{\cl}_{S\otimes_R M}.
\]

We do not know whether any closure operation mentioned above (including $\rof$ and $\wepf$ closures) has the persistence property for local morphisms between mixed-characteristic complete local domains. One can also ask the same question when the target ring is a complete local domain of characteristic $p$.

\bibliography{Closure_op_in_com_local_rings_of_mixed_char}
\bibliographystyle{halpha}

\end{document}